\documentclass[twoside,12pt]{amsart}
\usepackage[T1]{fontenc}

\usepackage[dvipsnames]{xcolor}
\usepackage{txfonts}
\usepackage{tikz}
\usetikzlibrary{arrows,automata, matrix, positioning, calc}
\usepackage{caption}
\usepackage[boxsize=0.3cm]{ytableau}
\usepackage{amsfonts}
\usepackage{amssymb}
\usepackage{amsthm}
\usepackage{amsmath}
\usepackage{mathrsfs} 
\usepackage{color}
\usepackage{graphicx}
\usepackage{caption}
\usepackage{float}
\usepackage{url}
\usepackage[all]{xy}
\usepackage{lscape}
\theoremstyle{plain}
\usepackage{array}
\usepackage{mathtools}
\usepackage{enumitem}
\usepackage[vcentermath]{youngtab}
\usepackage[margin=2.5cm]{geometry}

\usepackage{mynewcommands}

\begin{document}
% \date{\today}
\title{Heisenberg algebra, wedges and crystals}
\author{Thomas Gerber}

% \thanks{}

% \address{Lehrstuhl D f\"ur Mathematik, RWTH Aachen University,
% 52062 Aachen, Germany}
% 
% \email{gerber@math.rwth-aachen.de}

% \subjclass[2000]{17B37, 05E10, 20C08}
% \keywords{affine quantum group, Fock space, crystal graph, Heisenberg algebra, 
% Cherednik algebra, combinatorics}

\begin{abstract}
We explain how the action of the Heisenberg algebra on the space of $q$-deformed wedges
yields the Heisenberg crystal structure on charged multipartitions, by
using the boson-fermion correspondence and looking at the action of the Schur functions at $q=0$.
In addition, we give the explicit formula for computing this crystal in full generality.
\end{abstract}

\maketitle

% \tableofcontents

\markright{HEISENBERG ALGEBRA, WEDGES AND CRYSTALS}

\section{Introduction}

Categorification of representations of affine quantum groups has proved to be an important tool
for understanding many classic objects arising from modular group representation theory, among which
Hecke algebras and rational Cherednik algebras of cyclotomic type, and
finite classical groups.
More precisely, the study of crystals and canonical bases of the level $\ell$ Fock space representations $\cF_\bs$ of $\Ue$
gives answers to several classical problems in combinatorial terms.
In particular, we know that the $\Ue$-crystal graph of $\cF_\bs$ can be categorified in the following ways:
\begin{itemize}
\item by the parabolic branching rule for modular cyclotomic Hecke algebras \cite{Ariki1996},
when restricting to the connected component containing the empty $\ell$-partition,
\item by Bezrukavnikov and Etingof's parabolic branching rule for cyclotomic rational Cherednik algebras \cite{Shan2011},
\item by the weak Harish-Chandra modular branching rule on unipotent representations for finite unitary groups \cite{GerberHissJacon2015}, \cite{DudasVaragnoloVasserot2015},
for $\ell=2$ and $\bs$ varying.
\end{itemize}
In each case, the branching rule depends on some parameters that are explicitly determined by the parameters 
$e$ and $\bs$ of the Fock space.

Recently, there has been some important developments when Shan and Vasserot \cite{ShanVasserot2012}
categorified the action of the Heisenberg algebra on a certain direct sum of Fock spaces,
in order to prove a conjecture by Etingof \cite{Etingof2012}.
Losev gave in \cite{Losev2015} a formulation of Shan and Vasserot's results in terms of crystals,
as well as an explicit formula for computing it in a particular case.
Independently and using different methods, the author defined a notion of Heisenberg crystal for 
higher level Fock spaces \cite{Gerber2016}, that turns out to coincide with Losev's crystal.
An explicit formula was also given in another particular case, using level-rank duality.
Like the $\Ue$-crystal, the Heisenberg crystal gives some information at the categorical level.
In particular, it yields a characterisation of
\begin{itemize}
 \item the finite dimensional irreducible modules in the cyclotomic Cherednik category $\cO$ by \cite{ShanVasserot2012} and \cite{Gerber2016},
 \item the usual cuspidal irreducible unipotent modular representations of finite unitary groups \cite{DudasVaragnoloVasserot2016}.
\end{itemize}

This paper solves two remaining problems about the Heisenberg crystal.
Firstly, even though it originally arises from the study of cyclotomic rational Cherednik algebras
(it is determined by combinatorial versions of certain adjoint functors defined on the bounded derived category $\cO$),
the Heisenberg crystal has an intrinsic existence as shown in \cite{Gerber2016}.
Therefore, it is natural to ask for an algebraic construction of the Heisenberg crystal
which would be independent of any categorical interpretation.
This is achieved via the boson-fermion correspondence and the use of the Schur functions,
acting on Uglov's canonical basis of $\cF_\bs$.
This gives a new realisation of the Heisenberg crystal, analogous to Kashiwara's crystals
for quantum group.
Secondly, we give an explicit formula for computing the Heisenberg crystal in full generality.
This generalises and completes the particular results of \cite{Losev2015} and \cite{Gerber2016}.
This is done in the spirit of \cite{FLOTW1999} where formulas for the $\Ue$-crystal were explicited.

\medskip

The present paper has the following structure.
In Section 2, we start by introducing in detail several combinatorial objects
indexing the basis of the wedge space
(namely charged multipartitions, abaci and wedges)
and the different ways to identify them.
Then, we quickly recall some essential facts about the $\Ue$-structure of the 
Fock spaces $\cF_\bs$.
Section 3 focuses on the Heisenberg action on the wedge space $\cF_s$, seen as
a certain direct sum of level $\ell$ Fock spaces $\cF_\bs$.
In particular, we recall the definition of the Heisenberg crystal given in \cite{Gerber2016}.
Then, we give in Section 4 a solution to the first problem mentioned above.
More precisely, we recall the quantum boson-fermion correspondence and
fundamental facts about symmetric functions. 
Inspired by Shan and Vasserot \cite{ShanVasserot2012} and Leclerc and Thibon \cite{LeclercThibon2001},
we study the action of the Schur functions on the wedge space and use a result of Iijima \cite{Iijima2012}
to construct the Heisenberg crystal as a version of this action at $q=0$ (Theorem \ref{thmschur}),
resembling Kashiwara's philosophy of crystal and canonical bases.
Most importantly, by doing so, we bypass entirely Shan and Vasserot's original categorical construction.
Section 5 is devoted to the explicit computation of the Heisenberg crystal.
We introduce level $\ell$ vertical $e$-strips,
as well as the notion of good 
vertical $e$-strips by defining
an appropriate order.
The action of the Heisenberg crystal operators is then given in terms of
adding
good level $\ell$ vertical $e$-strips (Theorem \ref{thmheis}), which
is reminiscent of the explicit formula for the Kashiwara crystal
operators first given in \cite{JMMO1991} (see also \cite{FLOTW1999}).
We relate this result to other combinatorial procedures in the literature
answering in particular a question of Tingley \cite{Tingley2008}.
In addition, we give several examples of explicit computations.
Finally, we recall some useful facts about level-rank duality in appendix,
enabling the definition of the Heisenberg crystal.

\section{Higher level Fock spaces}\label{fockspace}

\subsection{Charged multipartitions and wedges} \label{mpwedge}

\subsubsection{Charged multipartitions}\label{multipart}

Fix once and for all elements $e,\ell\in\Z_{\geq 2}$ and $s\in\Z$.
An \textit{$\ell$-partition} (or simply \textit{multipartition}) is an $\ell$-tuple of partitions $\bla=(\la^1,\dots,\la^\ell)$. These will be represented using Young diagrams.
Denote $\Pi_\ell$ the set of $\ell$-partitions and $\Pi=\Pi_1$ the set of partitions.
Partitions will sometimes be denoted multiplicatively for convenience, e.g. $(2,1,1)=(2.1^2)={\ytableausetup{centertableaux, boxsize=2mm}\ydiagram{2,1,1}}\,$.

An \textit{$\ell$-charge} (or simply \textit{(multi)charge}) is an $\ell$-tuple of integers $\bs=(s_1,\dots,s_\ell)$.
We write $|\bs|=\sum_{j=1}^\ell s_j$.
We call \textit{charged $\ell$-partition} the data consisting of an $\ell$-partition $\bla$ and an $\ell$-charge $\bs$,
and denote it by $|\bla,\bs\rangle$.

For a box $\ga=(a,b,j)$ in the Young diagram of $\bla$ (where $a\in\Z$ is the row index of the box, $b\in\Z$ the column index and $j\in\{1,\dots,\ell\}$ the coordinate),
let $\fc(\ga)=b-a+s_j$, the \textit{content} of $\ga$.
We will represent $|\bla,\bs\rangle$ by filling the boxes of the Young diagram of $\bla$ with their contents.

\begin{exa}\label{exachargedmp}
Take $\ell=2$, $\bs=(-1,2)$ and $\bla=(2.1,1^2)$.
We have $s=-1+2=1$ and 
$$
|\bla,\bs\rangle = \left(\; \tiny \young(\moinsun0,\moinsdeux)\; , \; \young(2,1) \;\right).
$$
\end{exa}

In the following, we will only consider multicharges $\bs$ verifying $|\bs|:=s_1+\dots+s_\ell =s$.

\medskip

For a partition $\la$, let $\la'$ denote its conjugate, that is the partition obtained by swapping rows and columns in the Young diagram of $\la$.
We extend this to charged multipartitions by setting
$|\bla,\bs\rangle' = |\bla',\bs'\rangle$ where $\bla' = ((\la^\ell)',\dots,(\la^1)')$ and $\bs'=(-s_\ell,\dots,-s_1)$.

\subsubsection{Abaci representation} \label{abaci}

A charged $\ell$-partition $|\bla,\bs\rangle$ can also be represented by 
the $\Z$-graded $\ell$-abacus
$$\cA(\bla,\bs) = \left\{ (j, \la^j_k+s_j-k+1 ) \, ; \, j\in\{1,\dots,\ell\}, k\in\Z_{>1} \right\}$$
where $\la^j_k$ denotes the $k$-th part of $\la^j$.
In the rest of this paper, we will sometimes identify $|\bla,\bs\rangle$ with $\cA(\bla,\bs)$.

\begin{exa}\label{exaabacus}
With $|\bla,\bs\rangle$ as in Example \ref{exachargedmp},
we get the following corresponding abacus
\begin{equation*}
\cA(\bla,\bs)=\left\{
\begin{array}{clc} 
& \dots, (2,-5), (2,-4), (2,-3), (2,-2), (2,-1), (2,0) , (2,2), (2,3), & \\
& \dots, (1,-5), (1,-4), (1,-3), (1,-1), (1,1) &
\end{array} 
\right\}
\end{equation*}
which we picture as follows, by putting a (black) bead at position $(j,\la^j_k+s_j-k+1)$ where $k\in\Z_{>1}$
and $j\in\{1,\dots,\ell\}$ is the row index (numbered from bottom to top):
\begin{figure}[H]  \centering
\includegraphics[scale=1.2]{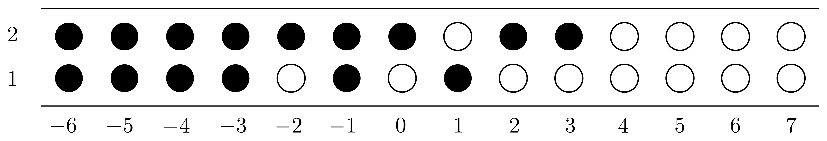}
\label{ab2}
\end{figure}
\end{exa}

Note that the conjugate a multipartition can be easily described on the abacus:
it suffices to rotate it by 180 degrees around the point of coordinates $(\frac{1}{2},\frac{\ell}{2})$
and swap the roles of the white and black beads.

\medskip

Using the abaci realisation of charged multipartitions, we define below a bijection
$$
\begin{array}{cccc}
\tau :
& \left\{|\la,s\rangle \,;\, \la\in\Pi_1\right\}
& \overset{\sim}{\lra}
& \left\{|\bla,\bs\rangle \,;\, \bla\in\Pi_\ell,|\bs|=s \right\}
\\
&
\cA(\la,s) 
& \longmapsto
& \cA(\bla,\bs),
\end{array}
$$
which can be seen as a twisted version of taking the $\ell$-quotient and the $\ell$-core of a partition, see \cite[Chapter 1]{Yvonne2005}.
However, unlike the usual $\ell$-quotient and $\ell$-core, $\tau$ and $\tau^{-1}$ will depend not only on $\ell$, but also on $e$.

\medskip

Writing down the Euclidean division first by $e\ell$ and then by $e$,
one can decompose any $n\in\Z$ as
$n = -z(n) e\ell + (y(n)-1)e + (x(n)-1)$ with $z(n)\in\Z$, $1\leq y(n) \leq \ell$ and $1\leq x(n)\leq e$.
We can then associate to each pair $(1,c)\in\{1\}\times\Z$ the pair $\tau(1,d)=(j,d)\in\{1,\dots,\ell\}\times\Z$
where
$$j=y(-c)\mand d=-(x(-c)-1)+ez(-c).$$

In particular, $\tau$ sends the bead in position $(1,c)$ into
the rectangle $z(-c)$, on the row $y(-c)$ and column $x(-c)$ (numbered from right to left within each rectangle).

The map $\tau$ is bijective and we can see $\tau^{-1}$ as the following procedure:
\begin{enumerate}
 \item Divide the $\ell$-abacus into rectangles of size $e\times\ell$, where the $z$-th rectangle ($z\in\Z$) contains the positions
 $(j,d)$ for all $1\leq j\leq \ell$ and $-e+1+ze\leq d \leq ze$.
 \item Relabel each $(j,d)$ by the second coordinate of $\tau^{-1}(j,d)$, see Figure \ref{ab9} for an example.
 \item Replace the newly indexed beads on a $1$-abacus according to this new labeling.
\end{enumerate}

\begin{figure}[H]  \centering
\includegraphics[scale=1.2]{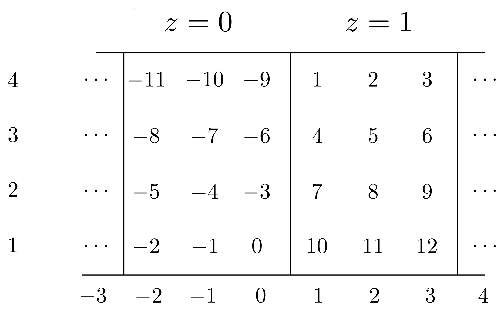}
\caption{Relabelling bead positions in the $\ell$-abacus according to $\tau^{-1}$, for $\ell=4$ and $e=3$.}
\label{ab9}
\end{figure}

Explicitly, $\tau$ and $\tau^{-1}$ verify the following formulas: 
$$ \arraycolsep=1.4pt\def\arraystretch{2.2}
\begin{array}{rcl}
\tau (1,c) & = & \left( \, -\left\lfloor\frac{-c}{e}\right\rfloor \mod \ell +1 \, , \,  (-c\mod e) + e \left\lfloor\frac{-c}{e\ell}\right\rfloor \, \right) \\
& = & \left( \frac{-c-(-c\mod e)}{e}\mod \ell +1 \, , \,  (-c\mod e) - e\left( \frac{-c-(-c\mod e\ell)}{e\ell}\right)\right)
\end{array}
$$
and
$$ \arraycolsep=1.4pt\def\arraystretch{2.2}
\begin{array}{rcl}
\tau^{-1} (j,d) & = & \left( 1 \, , \, -(-d\mod e) - e(j-1) + e\ell  \left\lfloor\frac{-d}{e}\right\rfloor \right) \\
% & = & \left( 1 \, , \, -(-d\mod e) - e(j-1) - \ell( -d-(-d\mod e)\right).
& = & \left( 1 \, , \, e(1-j) +\ell d - (\ell+1) (-d\mod e)\right).
\end{array}
$$

\begin{exa} We go back to Example \ref{exaabacus}.
The action of $\tau$ and $\tau^{-1}$ are represented below.
\begin{figure}[H]  \centering
\includegraphics[scale=1.2]{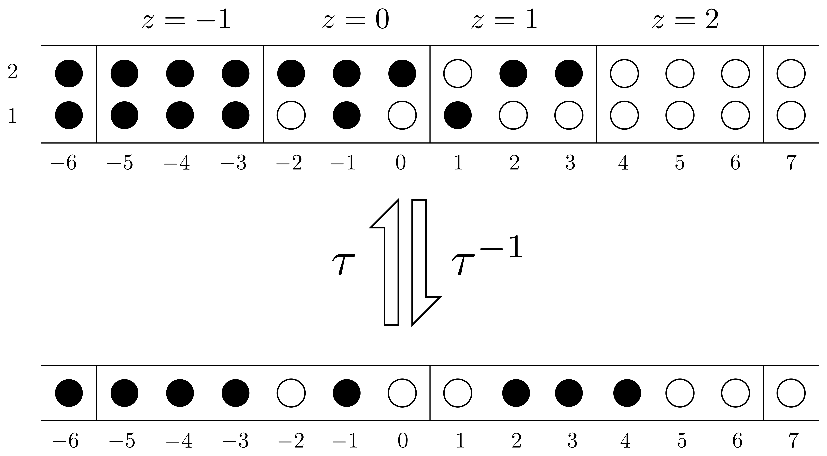}
\label{ab4}
\end{figure}
\end{exa}

\subsubsection{Wedges}\label{wedge}

Let $P(s)$ denote the set of sequences of integers 
$\bal=(\al_1,\al_2,\dots)$ such that $\al_k=s-k+1$ for $k$ sufficiently large,
and set $$P_{>}(s) = \left\{ (\al_1,\al_2,\dots)\in P(s) \, | \, \al_k>\al_{k+1} \text{ for all } k\in\Z_{>1} \right\}.$$

\begin{defi}\label{defiwp}
An \textit{elementary wedge} (respectively \textit{ordered wedge})
is a formal element 
$u_\bal = u_{\al_1} \w u_{\al_2}\w\dots$ where $\bal\in P(s)$
(respectively $\bal\in P_{>}(s)$). 
\end{defi}

\medskip

Let $|\bla,\bs\rangle$ be a charged $\ell$-partition with $|\bs|=s$,
and let $\la=(\la_1,\la_2,\dots)$ be the partition such that $\tau(|\la,s\rangle)=|\bla,\bs\rangle$.
Set $\bbe=(\be_1,\be_2,\dots)$ where $\be_k = \la_k+s-k+1$ for all $k\in\Z_{>1}$.
In other terms, $\bbe$ is the sequence of integers indexing the beads of $\cA(\la,s)$.

We clearly have $\bbe\in P_{>}(s)$, so we can consider the elementary wedge $u_\bbe$.
In fact, we will also identify $|\bla,\bs\rangle$ with $u_\bbe$.

\medskip

To sum up, we will allow the following identifications:
$$|\bla,\bs\rangle \llra \cA(\bla,\bs) \overset{\tau}{\llra} \cA (\la,s) \llra u_\bbe \llra |\la,s\rangle.$$

We will denote
$$B_\bs = \left\{ |\bla,\bs\rangle \, ; \, \bla\in\Pi_\ell \right\} \text{ for an $\ell$-charge $\bs$}$$
and
$$B_s = \left\{ |\la,s\rangle \, ; \, \la\in\Pi \right\} = \bigsqcup_{|\bs|=s} B_\bs.$$

\subsection{Fock space as $\Ue$-module}\label{uemodule}

Let $q$ be an indeterminate.

\subsubsection{The JMMO Fock space}\label{jmmofockspace}

Fix an $\ell$-charge $\bs$.
The \textit{Fock space} associated to $\bs$ is the $\Q(q)$-vector space $\cF_\bs$ 
with basis $B_\bs$.

\begin{thm}[Jimbo-Misra-Miwa-Okado \cite{JMMO1991}] \label{thmjmmo}
The space $\cF_\bs$ is an integrable $\Ue$-module of level $\ell$.
\end{thm}

The action of the generators of $\Ue$ depends on $\bs$ and $e$, and
is given explicitly in terms of addable/removable boxes, see e.g. \cite[Section 6.2]{GeckJacon2011}.
In turn, this induces a $\Ue$-crystal structure (also called Kashiwara crystal) \cite{Kashiwara1993}, usually encoded in a graph $\sB_\bs$, 
whose vertices are the elements of $B_\bs$, 
and with colored arrows representing the action of the Kashiwara operators.
An explicit (recursive) formula for computing this graph was first given in \cite{JMMO1991}: two vertices
are joined by an arrow if and only if one is obtained from the other by adding/removing a \textit{good} box, see \cite[Section 6.2]{GeckJacon2011}
for details.

\begin{figure}[H] \centering 
\begin{tikzpicture}[scale=0.7]

\node (a) at (5,6) {%
${ \tiny  \text{-} \, \text{-}  }$
};

\node (b1) at (2.5,4) {%
{${ \tiny \young(0)  \, \text{-}  }$}
}; 
\node (b2) at (7.5,4) {%
{${ \tiny \text{-}  \, \young(1)  }$}
};

\node (c1) at (1,2) {%
{${ \tiny \young(01)  \, \text{-}  }$}
}; 
\node (c2) at (4,2) {%
{${ \tiny \young(0,\moinsun) \, \text{-} }$}
}; 
\node(c3) at (6,2) {%
{${ \tiny \text{-} \, \young(12) }$}
}; 
\node (c4) at (9,2) {%
{${ \tiny \young(0)  \, \young(1)  }$}
}; 
\node(c5) at (12,2) {%
{${ \tiny \text{-} \, \young(1,0) }$}
};

\node(d1) at (0,0) {%
{${ \tiny \young(01) \, \young(1)  }$}
};
\node(d2) at (2,0) {%
{${ \tiny \young(012) \, \text{-} }$}
};
\node(d3) at (4,0) {%
{${ \tiny \young(01,\moinsun)  \, \text{-}  }$}
};
\node(d4) at (6,0) {%
{${ \tiny \text{-} \, \young(123)  }$}
};
\node(d5) at (8,0) {%
{${ \tiny \young(0)  \, \young(1,0)  }$}
};
\node(d6) at (10,0) {%
{${ \tiny \young(0) \, \young(12)  }$}
};

\node(d7) at (12,0) {%
{${ \tiny \text{-} \, \young(12,0) }$}
};
\node(d8) at (14,0) {%
{${ \tiny \young(0,\moinsun,\moinsdeux) \, \text{-}  }$}
};
\node(d9) at (16,0) {%
{${ \tiny \text{-} \, \young(1,0,\moinsun) }$}
};

\draw[->] (a) -- node [font=\tiny, left] {} (b1) ;
\draw[->] (a) -- node [font=\tiny, above] {} (b2) ;
\draw[->] (b1) -- node [font=\tiny,left] {} (c1) ;
\draw[->] (b1) -- node [font=\tiny, above] {} (c2) ;
\draw[->] (b2) -- node [font=\tiny, above] {} (c3) ;
\draw[->] (b2) -- node [font=\tiny, above] {} (c4) ;
\draw[->] (c5) -- node [font=\tiny, above] {} (d7) ;

\draw[->] (c1) -- node [font=\tiny, above] {} (d1) ;
\draw[->] (c1) -- node [font=\tiny, above] {} (d2) ;
\draw[->] (c2) -- node [font=\tiny, above] {} (d3) ;
\draw[->] (c3) -- node [font=\tiny, above] {} (d4) ;
\draw[->] (c4) -- node [font=\tiny, above] {} (d5) ;
\draw[->] (c4) -- node [font=\tiny, above] {} (d6) ;
\end{tikzpicture}

\caption{The beginning of the Kashiwara crystal graph $\sB_\bs$ of $\cF_\bs$ for $\ell=2$, $e=3$ and $\bs=(0,1)$}.
\label{uecrystal}
\end{figure}

It has several connected components, each of which parametrised by its
only source vertex, or \textit{highest weight vertex}.
The decomposition of this graph in connected components reflects the decomposition 
of $\cF_\bs$ into irreducible representations 
(which exists because $\cF_\bs$ is integrable according to Theorem \ref{thmjmmo}).

\subsubsection{Uglov's wedge space}\label{uglovfockspace}

Denote $\cF_s$ the $\Q(q)$-vector space spanned by the elementary wedges, and subject to the straightening relations
defined in \cite[Proposition 3.16]{Uglov1999}.
This is called the \textit{space of semi-infinite $q$-wedges}, or simply the \textit{wedge space},
and the elements of $\cF_s$ will be called \textit{wedges}.

Using the straightening relations, a $q$-wedge can be expressed as a $\Q(q)$-linear combination of 
ordered wedges. In fact, one can show (see \cite[Proposition 4.1]{Uglov1999}) that 
the set of ordered wedges $B_s$ forms a basis of $\cF_s$.

\begin{rem}\
\begin{enumerate}
\item The terminology ``wedge space'' is justified by the original construction of this vector space in the level one case by
Kashiwara, Miwa and Stern \cite{KashiwaraMiwaStern1995}. 
In this context, one can first construct a quantum version of the usual $k$-fold
wedge product (or exterior power) $\La^k V$, where $V$ is the natural $\Ue$-representation 
(an affinisation of $\C^e\otimes\C^\ell$).
The space $\cF_s$ is then defined as the projective limit (taking $k\ra\infty$) of $\La^kV$.
In the higher level case, the analogous construction was achieved by Uglov \cite{Uglov1999}.
\item Using the bijection between ordered wedges and partitions charged by $s$, one can
see $\cF_s$ as a level $1$ Fock space, whence the notation.
In fact, it is sometimes called the \textit{fermionic} Fock space.
% To avoid confusion, we will prefer the terminology ``wedge space'', so that
% the term ``Fock space'' refers only to level $\ell$ representations of $\Ue$.
\end{enumerate}
\end{rem}

\begin{thm}\label{thmwedgeuemod}
The wedge space $\cF_s$ is an integrable $\Ue$-module of level $\ell$.
\end{thm}

\proof
The identification
$B_s=\left\{u_\bbe \in \cF_s \,|\, \be\in P_{>}(s) \right\}
= 
\left\{| \bla,\bs\rangle \in \cF_\bs \, ;\, |\bs|=s \right\}
= \bigsqcup_{|\bs|=s}B_\bs
$
of Section \ref{wedge} yields the $\Q(q)$-vector space decomposition
$$\cF_s = \bigoplus_{|\bs|=s}\cF_\bs.$$
By Theorem \ref{jmmofockspace}, this decomposition still holds as integrable $\Ue$-module.
\endproof

\section{The Heisenberg action}\label{heis}

It turns out that the wedge space has some additional structure, namely that of a $\cH$-module, where
$\cH$ is the (quantum) Heisenberg algebra.
This has been first observed by Uglov \cite{Uglov1999}, generalising some results of Lascoux, Leclerc and Thibon \cite{LLT1997}.
Let us recall the definition of this algebra.

\subsection{The action of the bosons}\label{boson}

Let us start by recalling the definition of the quantum Heisenberg algebra.

\begin{defi}\label{defheis}
The \textit{(quantum) Heisenberg algebra} is the unital $\Q(q)$-algebra $\cH$ with generators $ p_m \, , \, m\in \Z^\times$ and
defining relations 
$$ [p_m,p_{m'}] = \de_{m,-m'} m \frac{1-q^{-2me}}{1-q^{-2m}}\times\frac{1-q^{2m\ell}}{1-q^{2m}}$$
for $m,m'\in\Z^\times$.
The elements $p_m$ are called \textit{bosons}.
\end{defi}

Note that this is a $q$-deformation of the usual Heisenberg algebra with relations $[p_m,p_{m'}]=\de_{m,-m'}m$.

\begin{thm}\label{thmboson}
The formula
$$p_m(u_\bbe) = \sum_{k\geq 1} u_{\be_1} \w \dots \w u_{\be_{k-1}} \w u_{\be_{k-e\ell m}}\w u_{\be_{k+1}}\w \dots \text{\hspace{1cm} for $u_\be\in\cF_s$ and $m\in\Z^\times$}$$
 endows $\cF_s$ with the structure of $\cH$-module.
\end{thm}

For a proof, we refer to \cite[Proposition 4.4 and 4.5]{Uglov1999}.
This is quite technical and is done in two distinct steps, the second of which
requires the notion of \textit{asymptotic} wedge. This will be crucial in Section \ref{cbschur}.

\begin{cor}\label{corboson}
The action of the bosons preserves the level $\ell$ Fock spaces $\cF_\bs$ for $|\bs|=s$.
\end{cor}

One can use the explicit action	of $p_m$ on $\cF_s$ to show that it preserves the $\ell$-charges $\bs$,
or rely on Uglov's argument \cite[Section 4.3]{Uglov1999}.
As a consequence of this corollary, $\cF_s=\bigoplus_{|\bs|=s} \cF_\bs$ also holds as $\cH$-module decomposition.

\subsection{Some notations and definitions}

For two partitions $\la$ and $\mu$, denote $\la + \mu$ the reordering of $(\la_1,\mu_1,\la_2,\mu_2,\dots)$
and $k\la$ the partition $(\la_1^k.\la_2^k.\dots)$.
Extend these notations to $\ell$-partitions by doing these operations coordinatewise.

\begin{rem}\label{remadd}
This is not the usual notation for adding partitions or multiplying by an integer. 
One recovers the usual one by conjugating
\end{rem}
  
\ytableausetup{centertableaux}

\begin{exa}
We have
$$\ydiagram{4,1,1} + 2 \, \ydiagram{2,1} =\ydiagram{4,1,1} + \ydiagram{2,2,1,1} = \ydiagram{4,2,2,1,1,1,1}\,.$$
\end{exa}

We also define the notion of asymptotic charges, which will mostly be useful in Section \ref{cbschur}.

\begin{defi}\label{defasymptotic}
A wedge $u_\bbe\in\cF_s$ is called \textit{$j_0$-asymptotic} if there exists $j_0\in\{1,\dots,\ell\}$ 
such that the corresponding charged multipartition $|\bla,\bs\rangle$ verifies
$|s_j-s_{j_0}|\geq |\bla|$ for all $j\neq j_0$.
\end{defi}

\subsection{The Heisenberg crystal}

A notion of crystal for the Heisenberg algebra, or $\cH$-crystal, has been indepedently introduced by Losev \cite{Losev2015}
and the author \cite{Gerber2016}.
Explicit formulas for computing this crystal have been given in some particular cases: 
asymptotic in \cite{Losev2015} and \textit{doubly highest weight} in \cite{Gerber2016}.

\medskip

Recall the definition of the $\cH$-crystal according to \cite{Gerber2016}.
This requires the crystal level-rank duality exposed in Appendix.

\begin{defi}\label{defdhw} Let $|\bla,\bs\rangle \in B(s)$,
which we identify with its level-rank dual charged $e$-partition. 
We call $|\bla,\bs\rangle$ a \textit{doubly highest weight vertex} if
it is a highest weight vertex simultaneously  in the $\Ue$-crystal and in the $\Ul$-crystal.
\end{defi}

Some important properties of doubly highest weight vertices are exposed in \cite[Section 5.2]{Gerber2016}.
In particular, an element $|\bla,\bs\rangle$ is a doubly highest weight vertex if and only if it has a totally $e$-periodic
$\ell$-abacus and a totally $\ell$-periodic $e$-abacus,
according to a result by Jacon and Lecouvey \cite{JaconLecouvey2012}, see the definition therein.
Moreover, every bead of a given period encodes a same part size in $\bla$, and one can define
the partition $\ka=\ka|\bla,\bs\rangle$ as $(\ka_1,\ka_2\,\dots)$
where $\ka_k$ is the part encoded by the $k$-th period.

\begin{defi}\label{defbsigma}
Let $\si\in\Pi$. 
The Heisenberg crystal operator $\tb_\si$ (respectively $\tb_{-\si}$) is the uniquely determined map $B_s \ra  B_s \, |\bla,\bs\rangle \mapsto |\bmu,\bs\rangle$ 
(respectively $B_s\ra B_s\sqcup\{0\}$) such that
\begin{enumerate}
 \item if $|\bla,\bs\rangle$ is a doubly highest weight vertex, then $|\bmu,\bs\rangle$ is obtained from
 $|\bla,\bs\rangle$ by shifting the $k$-th period of $\ell$-abaci of $|\bla,\bs\rangle$ by $\si_k$ steps to the right (respectively to the left when possible,
and $\tb_{-\si}|\bla,\bs\rangle=0$ otherwise),
\item it commutes with the Kashiwara crystal operators of $\Ue$ and of $\Ul$.
\end{enumerate}
\end{defi}

\begin{defi}\label{defHhw} 
We say that $|\bla,\bs\rangle\in B_s$ is a highest weight vertex for $\cH$ if
$\tb_{-\si}|\bla,\bs\rangle=0$ for all $\si\in\Pi$.
\end{defi}

Note that each $\tb_{\pm\si}$ is well defined since (2) allows to define $\tb_{\pm\si}|\bla,\bs\rangle$ even when
$|\bla,\bs\rangle$ is not a doubly highest weight vertex.
Moreover, \cite[Theorem 7.6]{Gerber2016} claims that in the asymptotic case, the Heisenberg crystal operators 
coincide with the combinatorial maps introduced by Losev \cite{Losev2015}.
Note also that the definition of $\ka$ for doubly highest weight vertices induces (by (2) of Definition \ref{defbsigma}) 
a surjective map $\ka : B_s  \ra \Pi$
such that $\tb_{-\ka|\bla,\bs\rangle}|\bla,\bs\rangle$ is a highest weight vertex for $\cH$ for all $\bla\in\Pi_\ell$.

\medskip

In order to have a description of the $\cH$-crystal similar to the $\Ue$-crystal, see Figure \ref{uecrystal},
we wish to define the $\cH$-crystal as a graph whose arrows have minimal length.
Therefore, we define the following maps 
$$\tb_{1,c} = \tb_{\eta} \tb_{-\ka} \mand \tb_{-1,d} = \tb_{\theta} \tb_{-\ka}$$
where $\eta = \ka\sqcup\{\ga\}$ with $\ga = (a,b)$ being the addable box of $\ka$ verifying $b-a=c$ and
where $\theta = \ka\backslash\{\ga\}$ with $\ga = (a,b)$ being the removable box of $\ka$ verifying $b-a=d$,

\begin{defi}\label{defHcrystal}
The $\cH$-crystal of the wedge space $\cF_s$ is the graph $\sC_s$ with
\begin{enumerate}
 \item vertices: all charged $\ell$-partitions $|\bla,\bs\rangle$ with $\bla\in\Pi_\ell$ and $|\bs|=s$.
 \item arrows : $|\bla,\bs\rangle \overset{c}{\lra} |\bmu,\bs\rangle$ if $|\bmu,\bs\rangle=\tb_{1,c}|\bla,\bs\rangle$.
\end{enumerate}
\end{defi}

\begin{rem}
By definition, each $\tb_{\pm 1,c}$ is a composition of maps $\tb_{\pm\si}$.
Conversely, each $\tb_\si$ is a composition of maps $\tb_{1,c}$, see \cite[Formula (6.17)]{Gerber2016}.
We will also call the maps $\tb_{1,c}$  Heisenberg crystal operators.
\end{rem}

\begin{prop}\label{propHcrystal}\
\begin{enumerate}
 \item The $\cH$-module decomposition $\cF_s=\bigoplus_{|\bs|=s}\cF_\bs$ induces a decomposition of $\sC_s$.
 \item Each connected component of $\sC_s$ is isomorphic to the Young graph.
 \item A vertex $|\bla,\bs\rangle$ is source in $\sC_s$ if and only if $|\bla,\bs\rangle$ is a highest weight vertex for $\cH$.
 \item The depth of an element $|\bla,\bs\rangle$ in $\sC_s$ is $|\ka|\bla,\bs\rangle|$.
\end{enumerate}
\end{prop}

\begin{proof}
We know from Definition \ref{defbsigma} that the action of the crystal Heisenberg operators preserves the multicharge, proving (1).
In particular, there is a notion of $\cH$-crystal for $\cF_\bs$, which we denote $\sC_\bs$.
Moreover, the $\cH$-crystal is characterised by being the preimage of a $\Uinf$-crystal on the set of partitions
under a certain bijection, depending on $\ka$, see \cite[Remark 6.16]{Gerber2016}.
The $\Uinf$-crystal is exactly the Young graph on partitions, see Figure \ref{younggraph}, where the arrows are colored by the contents of the added boxes,
which proves (2).
In fact, the bijection from the Young graph to a given connected component, parametrised by its source vertex $|\overline{\bla},\bs\rangle$
is given by $\si \mapsto \tb_\si|\overline{\bla},\bs\rangle$, and its inverse is $|\bla,\bs\rangle \mapsto \ka|\bla,\bs\rangle$.
Point (3) is clear by definition, and (4) follows from the definition of $\ka$.
\end{proof}

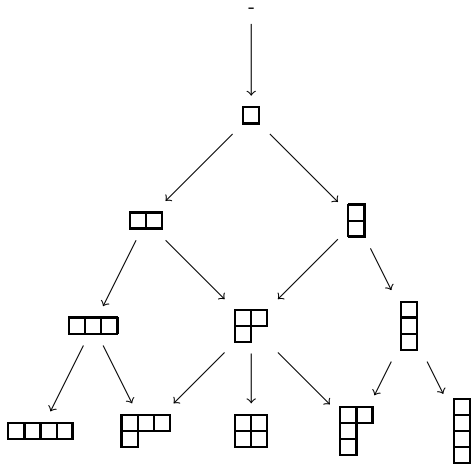
\begin{figure}[h] \centering
\begin{tikzpicture}[scale=0.7]

\node (a) at (4,8) {%
${ \tiny  \text{-} }$
};

\node (b) at (4,6) {%
{${ \tiny \ydiagram{1}   }$}
};

\node (c1) at (2,4) {%
{${ \tiny \ydiagram{2} }$}
}; 
\node (c2) at (6,4) {%
{${ \tiny \ydiagram{1,1} }$}
};

\node(d1) at (1,2) {%
{${ \tiny \ydiagram{3}  }$}
};
\node(d2) at (4,2) {%
{${ \tiny \ydiagram{2,1} }$}
};
\node(d3) at (7,2) {%
{${ \tiny \ydiagram{1,1,1} }$}
};

\node(e1) at (0,0) {%
{${ \tiny \ydiagram{4}  }$}
};
\node(e2) at (2,0) {%
{${ \tiny \ydiagram{3,1} }$}
};
\node(e3) at (4,0) {%
{${ \tiny \ydiagram{2,2} }$}
};
\node(e4) at (6,0) {%
{${ \tiny \ydiagram{2,1,1}  }$}
};
\node(e5) at (8,0) {%
{${ \tiny \ydiagram{1,1,1,1} }$}
};

\draw[->] (a) -- node [font=\tiny, left] {} (b) ;
\draw[->] (b) -- node [font=\tiny,left] {} (c1) ;
\draw[->] (b) -- node [font=\tiny, above] {} (c2) ;
\draw[->] (c1) -- node [font=\tiny, above] {} (d1) ;
\draw[->] (c1) -- node [font=\tiny, above] {} (d2) ;
\draw[->] (c2) -- node [font=\tiny, above] {} (d3) ;
\draw[->] (c2) -- node [font=\tiny, above] {} (d2) ;

\draw[->] (d1) -- node [font=\tiny, above] {} (e1) ;
\draw[->] (d1) -- node [font=\tiny, above] {} (e2) ;
\draw[->] (d2) -- node [font=\tiny, above] {} (e2) ;
\draw[->] (d2) -- node [font=\tiny, above] {} (e3) ;
\draw[->] (d2) -- node [font=\tiny, above] {} (e4) ;
\draw[->] (d3) -- node [font=\tiny, above] {} (e4) ;
\draw[->] (d3) -- node [font=\tiny, above] {} (e5) ;
\end{tikzpicture}

\caption{The beginning of the Young graph.}
\label{younggraph}
\end{figure}

\section{Canonical bases and Schur functions}\label{cbschur}

\subsection{The boson-fermion correspondence}\label{bosonfermion}

Denote $\La$ the algebra of symmetric functions,
that is, the projective limit of the $\Q(q)$-algebras of symmetric polynomials in
finitely many indeterminates \cite[Chapter 1]{Macdonald1998}:
$$\La = \Q(q)[X_1,X_2,\dots]^{\fS_\infty}.$$

The space $\La$ has several natural linear bases, among which:
\begin{itemize}
\item the monomial functions $\left\{ M_\si \, ; \, \si\in\Pi\right\}$ where 
 $M_\si=\sum_{\pi} X_1^{\pi_1}X_2^{\pi_2}\dots$ where the sum runs over all permutations $\pi$ of $\si$,
 \item the complete functions $\left\{ H_\si \, ; \, \si\in\Pi\right\}$ where  
 $H_\si=H_{\si_1} H_{\si_2}\dots$ and $H_m=\sum_{k_1\leq\dots\leq k_m} X_{k_1}\dots X_{k_m}$
 for $r\in\N$,
 \item the power sums $\left\{ P_\si \, ; \, \si\in\Pi\right\}$ where  
 $P_\si=P_{\si_1} P_{\si_2}\dots$ and $P_m=\sum_{k\geq 1} X_{k}^r$  for $r\in\N$,
 \item the Schur functions $\left\{ S_\si \, ; \, \si\in\Pi\right\}$ where  
 $S_\si=\sum_{\pi\in\Pi}K_{\si,\pi}M_\pi$ where $K_{\si,\pi}$ are the Kostka numbers, see \cite[Chapter 7]{Stanley2001}.
\end{itemize}
The expansion of $H_m$ in the basis of the power sums is given by 
$$H_m=\sum_{\substack{\pi\in\Pi \\ |\pi|=m}}\frac{1}{z_\pi}P_\pi,$$
where $z_\pi = \Pi_{k>0} k^{\al_k}\al_k$ with $\al_k=\pi'_k-\pi'_{k+1}$, for all $\pi\in\Pi$.
Moreover, by duality, the Kostka numbers also appear as the coefficients of the complete functions in the basis of the Schur functions:
$$H_\si=\sum_{\pi\in\Pi}K_{\pi,\si}S_\pi.$$

By a result of Miwa, Jimbo and Date \cite{MiwaJimboDate2000}, 
there is a vector space isomorphism
$$
\begin{array}{ccc}
\cF_s & \overset{\sim}{\lra} & \La \\ 
|\la,s\rangle & \longmapsto & S_\la,
\end{array}
$$
called the \textit{boson-fermion} correspondence.
In fact, when refering to the symmetric realisation $\La$, one sometimes uses the term \textit{bosonic} Fock space,
as opposed to the fermionic, antisymmetric definition of $\cF_s$.
The following result is \cite[Section 4.2 and Proposition 4.6]{ShanVasserot2012}.

\begin{prop}\label{propheis}
 There is a $\cH$-module isomorphism $\cF_s\simeq \La$, 
where, at $q=1$, the action of $p_m$ on $\cF_s$ 
is mapped to the multiplication by $P_m$ on $\La$.
\end{prop}

In general, the action of $p_m$ is mapped to a $q$-deformation of the multiplication by $P_m$, 
and $p_{-m}$ to a $q$-deformation of the derivation with respect to $P_m$, see \cite[Section 5]{LLT1997} and \cite[Section 5.1]{Uglov1999}.

\subsection{Action of the Schur functions}

In \cite{LeclercThibon2001}, Leclerc and Thibon studied the action of the Heisenberg algebra on level $1$ Fock spaces
in order to give an analogue of Lusztig's version \cite{Lusztig1989} of the Steinberg tensor product theorem.
Their idea was to use a different basis of $\La$ to compute the Heisenberg action in a simpler way,
namely that of Schur functions.
This result has been generalised to the level $\ell$ case by Iijima in a particular case \cite{Iijima2012}.

Independently, the Schur functions have been used by Shan and Vasserot 
to categorify the Heisenberg action in the context of Cherednik algebras.
More precisely, they constructed a functor on the category $\cO$ for cyclotomic rational 
Cherednik algebras corresponding to the multiplication by a Schur function on the bosonic Fock space $\La$ \cite[Proposition 5.13]{ShanVasserot2012}.

The aim of this section is to use Iijima's result to recover in a direct, simple way the results of \cite{Losev2015} and \cite{Gerber2016}
and, by doing so, bypassing Shan and Vasserot's categorical constructions.

\subsubsection{Canonical bases of the Fock space}

In the early nineties, Kashiwara and Lusztig have independently introduced the notion of canonical bases
for irreducible highest weight representations of quantum groups, see e.g. \cite{Kashiwara1993}.
They are characterised by their invariance under a certain involution.
Uglov \cite{Uglov1999} has
proved an analogous result for the Fock spaces $\cF_\bs$, even though it is no longer irreducible.
We recall the definition of the involution on $\cF_s$ and Uglov's theorem.

For any $r\in\N$ and $t_1,\dots,t_r\in\Z^r$, let
$\iota(t_1,\dots,t_r)=\sharp \left\{ \, (k,k') \,\, | \,\, k<k' \text{ and } t_k=t_{k'} \, \right\}$
that is, the number of repetitions in $(t_1,\dots,t_r)$.

\begin{defi}\label{defibar}
The bar involution is the $\Q(q)$-vector space automorphism 
$$
\begin{array}{lcl}
\cF_s & \lra & \cF_s \\
q & \longmapsto & \overline{q}=q^{-1} \\
u_\be & \longmapsto & \overline{u_\be}
\end{array}
$$
with
$$\overline{u_\be} = \overline{u_{\be_1}\w\dots \w u_{\be_r}}\w u_{\be_{r+1}}\w\dots
= (-q)^{\iota(y_1,\dots,y_r)} q^{\iota(x_1,\dots,x_r)} (u_{\be_r}\w \dots \w u_{\be_1})\w u_{\be_{r+1}}\w\dots
$$
where $x_k = x(\be_k)$ and $y_k = y(\be_k)$ for all $k=1,\dots,r$ according to the notations of Section \ref{abaci}.
\end{defi}

The bar involution behaves nicely on the wedge space, in particular
it preserves the level $\ell$ Fock spaces $\cF_\bs$ for $|\bs|=s$,
and commutes with the bosons $p_m$, this is \cite[Section 4.4]{Uglov1999}.
We can now state Uglov's result, that is derived from the fact that the 
matrix of the bar involution is unitriangular.

\begin{thm}\label{thmcanonicalbases}
Let $\bs\in\Z^\ell$ such that $|\bs|=s$.
There exist unique bases \mbox{$\cG^+=\left\{ G^+(\bla,\bs) \, ; \, \bla\in\Pi_\ell \right\}$} and 
\mbox{$\cG^-=\left\{ G^-(\bla,\bs) \, ; \, \bla\in\Pi_\ell \right\}$} of $\cF_\bs$ such that, for $\flat\in\{+,-\}$,
\begin{enumerate}
 \item $\overline{G^\flat(\bla,\bs)}= G^\flat(\bla,\bs)$
 \item $G^\flat(\bla,\bs) = |\bla,\bs\rangle \mod q^{\flat1}\cL^\flat$ 
 where $\cL^{\flat}=\bigoplus_{\bla\in\Pi_\ell}\Q[q^{\flat1}]|\bla,\bs\rangle$.
\end{enumerate}
\end{thm}

This result is compatible with Kashiwara's crystal theory.
More precisely, each integrable irreducible highest weight $\Ue$-representation is contained in $\cF_s$ for some $s$, 
by taking the span of the vectors $|\bemptyset,\bs\rangle$ for $|\bs|=s$,
and it is proved in \cite[Section 4.4]{Uglov1999} that the bases $\cG^\flat$
restricted to $\Ue |\bemptyset,\bs\rangle$ coincide with Kashiwara's lower and upper canonical bases.
Therefore, we also call $\cG^\flat$ the 
(lower or upper if $\flat=-$ or $+$ respectively) canonical basis of $\cF_\bs$.

\subsubsection{Schur functions in the asymptotic case}

Define the following operators on $\cF_s$
$$h_m=\sum_{|\pi|=m}\frac{1}{z_\pi}p_\pi, \quad h_\si= h_{\si_1}h_{\si_2}\dots
\mand
s_\si=\sum_{\pi\in\Pi}K_{\pi,\si}^{-1}h_\pi.
$$ where
$K_{\pi,\si}^{-1}$ are the inverse Kostka numbers, that is,
the entries of the inverse of the matrix of Kostka numbers. 
By Proposition \ref{propheis}, at $q=1$, 
the action of $h_\si$ (respectively $s_\si$) corresponds to the multiplication by
a complete function (respectively Schur function) on $\La$ through the boson-fermion correspondence.

\begin{thm}\label{thmschur}\
\begin{enumerate}
\item The operators $s_\si$ induce maps $\ts_\si$ on $B_s$, preserving $B_\bs$ for $|\bs|=s$.
\item Let $|\bla,\bs\rangle\in\cF_\bs$ be $j_0$-asymptotic.
We have $\ts_\si |\bla,\bs\rangle = |\bmu,\bs\rangle $ with $\bmu = \bla  + e\bsi$ where $\si_j=\emptyset$ if $j\neq j_0$ and $\si_{j_0}=\si$,
provided $|\bmu,\bs\rangle$ is still $j_0$-asymptotic.
In particular, $\ts_\si$ coincides with the Heisenberg crystal operator $\tb_\si$ in this case.
 \end{enumerate}
\end{thm}

\begin{rem}
When $\bla=\bemptyset$, this says that $\ts_\si$ shifts the $e$ rightmost beads in the $j_0$-th runner of $\cA(\bla,\bs)$,
and one recovers Losev's result, see for instance \cite[Example 7.3]{Gerber2016}.
\end{rem}

\proof
In a general manner, we can identify a crystal map $B_s\ra B_s$ with an operator on the wedge space $\cF_s$
mapping an element of $\cG^\flat$ to an element of $\cG^\flat$, for $\flat\in\{+,-\}$.
Indeed, an element $|\bla,\bs\rangle\in B_s$ can be identified with $G^\flat(\bla,\bs)$.
Another way to see this is to look at the action on $G^\flat(\bla,\bs)$ and put $q=0$ or $q=\infty$ respectively 
in the resulting vector. This provides the identification because of Condition (2) of Theorem \ref{thmcanonicalbases}.
As a matter of fact, the action of $s_\si$ on a canonical basis element turns out to have the desired form.
Indeed, by \cite[Theorem 4.12]{Iijima2012}, we know that $s_\si :\cF_s\ra\cF_s$ verifies
$$s_\si G^+(\bla,\bs) = G^+(\bmu,\bs)$$
where $\bmu=\bla + e\bsi$ (with $\bsi$ as in the statement of Theorem \ref{thmschur}) provided:
\begin{itemize}
 \item $\bla^j$ is $e$-regular for all $j=1\dots,\ell$
 \item $|\bla,\bs\rangle$ and $|\bmu,\bs\rangle$ are $j_0$-asymptotic.
\end{itemize}
Here, Iijima's original statement has been twisted by conjugation, because the level-rank duality 
used in his result is the reverse of that of the present paper.
Thus, the notion of $e$-restricted multipartition is replaced by $e$-regular.
Accordingly, we use the upper canonical basis instead of the lower one.
This statement can be extended to an arbitrary $j_0$-asymptotic element $G^+(\bla,\bs)\in\cF_s$, that is, 
without the restriction that each $\la^j$ is $e$-regular. To do this,
for all $\bla\in\Pi_\ell$, let $\bla=\tbla+e\bpi$ where $\tbla^j$ is $e$-regular for all $j=1,\dots,\ell$
and $\pi_j=\emptyset$ if $j\neq j_0$ and $\pi_{j_0}=\pi$.
Then, for all $\si\in\Pi$,
$$s_\si G^+(\bla,\bs) = s_{\tau}G^+(\tbla,\bs)$$
where $\tau=(\si'+\pi')'$, which we can compute by the preceding formula.
Note that $\tau$ is the common addition of $\si$ and $\pi$, see Remark \ref{remadd}.
For a general $j_0$-asymptotic vector $|\bla,\bs\rangle$, we write again $G^+(\bmu,\bs)=s_\si G^+(\bla,\bs)$.
As explained in the beginning of the proof, this induces a crystal map $\ts_\si:B_s \ra B_s$,
by additionnaly requiring that $\ts_\si$ commutes with the Kashiwara crystal operators.
In fact, when $|\bla,\bs\rangle$ is $j_0$-asymptotic, the formula for $\ts_\si$
is precisely the formula for the Heisenberg crystal operator given in \cite{Losev2015}
(provided again that one twists by conjugating),
which coincides with $\tb_\si$ by \cite[Theorem 7.6]{Gerber2016}.
This completes the proof.
\endproof

As explained in the proof, with this approach, $\ts_\si|\bla,\bs\rangle$ is identified with $s_\si G^+(\bla,\bs)$.
The map $s_\si$ being an actual operator (on the vector space $\cF_s$), 
this justifies the terminology ``operator'' for the maps $\tb_\si : B_s\ra B_s$,
thereby completing the analogy with the Kashiwara crystal operators.

\section{Explicit description of the Heisenberg crystal}\label{heiscrystal}

In this section, we give the combinatorial formula for computing the Heisenberg crystal in full generality.
This completes the results of \cite{Losev2015}, where the asymptotic case is treated, and of \cite{Gerber2016}
where the case of doubly highest weight vertices (in the level-rank duality) is treated, see Appendix for details.

\subsection{Level $\ell$ vertical strips}

In the spirit of \cite{JMMO1991} and \cite{FLOTW1999}, we will express the action of the Heisenberg crystal operators
in terms of adding/removing certain vertical strips.

\medskip

For a given charged multipartition $|\bla,\bs\rangle$,
we denote 
$$
\sW_1(\bla,\bs) = \left\{ (a,b,j)\in\Z_{>0}\times\Z_{>1}\times\{1,\dots,\ell\}  \; | \; (a,b,j) \notin\bla \mand
(a,b-1,j)\in \bla \right\}
$$
$$
\sW_2(\bla,\bs)=  \left\{ (a,1,j)\in\Z_{>0}\times\{1\}\times\{1,\dots,\ell\}  \; | \; (a,1,j) \notin\bla \right\}
$$
and $\sW(\bla,\bs) = \sW_1(\bla,\bs) \sqcup \sW_2(\bla,\bs)$,
so that $\sW(\bla,\bs)$ is the set of boxes directly to the right of $\bla$ (considering that it has infinitely many parts of size zero).

\begin{defi}\label{defvertstrip}
Let $|\bla,\bs\rangle$ be a charged $\ell$-partition.
\begin{enumerate}
 \item A \textit{(level $\ell$) vertical $e$-strip} is a sequence of $e$ boxes $\ga_1=(a_1,b_1,j_1),\dots,\ga_e=(a_e,b_e,j_e)$ such that no horizontal domino appears,
 i.e.  there is no $1\leq i \leq e$ such that $a_{i+1}=a_i$ and $j_{i+1}=j_i$.
 Moreover, a vertical $e$-strip is called \textit{admissible} if:
\begin{enumerate}
 \item The contents of $\ga_1,\dots,\ga_e$ are consecutive, say $\fc(\ga_i)=\fc(\ga_{i+1})+1$ for all $1\leq i \leq e$.
 \item For all $1\leq i<i' \leq e$, we have ${j_i}\geq {j_{i'}}$.
\end{enumerate}
\item The admissible vertical $e$-strips contained in $\sW(\bla,\bs)$ 
% (respectively $\sW^+(\bla,\bs)$ and $\sW^-(\bla,\bs)$)
are denoted $\sV(\bla,\bs)$ 
% (respectively $\sV^+(\bla,\bs)$ and $\sV^-(\bla,\bs)$)
.
The elements of $\sV(\bla,\bs)$ are called the \textit{(admissible) vertical $e$-strips of $|\bla,\bs\rangle$}.
\item A vertical $e$-strip $X$ of $|\bla,\bs\rangle$ is called \textit{addable} 
% (respectively \textit{removable})
if $X\cap \bla = \emptyset$ and $\bla \sqcup X$ is still an $\ell$-partition 
% (respectively $X\subseteq \bla$ and $\bla \backslash X$ is still an $\ell$-partition)
.
\end{enumerate}
\end{defi}

\begin{rem}\label{remverticalstrip}
This is a generalisation to multipartitions of the usual notion of vertical strips, see for instance \cite[Chapter I]{Macdonald1998}.
\end{rem}

\begin{exa}\label{exaverticalstrip}
Let $\ell=3$, $e=4$ and 
$$|\bla,\bs\rangle=|(4.2,2,2^2.1^2),(1,4,6)\rangle={\tiny\left(\,\young(1234,01) \, , \, \young(45) \, \, , \, \young(67,56,4,3) \, \right) }.$$
Then 
$\sV(\bla,\bs)$ consists of 
\begin{center}
\begin{tabular}{ccl}
 $X_{1}=( (1,3,3) , (2,3,3) ,(1,3,2), (1,5,1) )$ & with respective contents & $8,7,6,5$ \, ,  \\
 $X_{2}=( (3,2,3) , (4,2,3), (2,1,2), (2,3,1) )$ & with respective contents & $5,4,3,2$ \, ,\\
 $X_{3}=( (3,2,3) , (4,2,3), (2,1,2), (3,1,2) )$ & with respective contents & $5,4,3,2$ \, ,\\
 
 $X_{4}=( (4,2,3), (2,1,2), (3,1,2), (4,1,2) )$ & with respective contents & $4,3,2,1$ \, ,\\
 $X_{5}=( (2,1,2) , (3,1,2), (4,1,2), (5,1,2) )$ & with respective contents & $3,2,1,0$ \, ,\\
 
 $X_{6}=( (3,1,2) , (4,1,2), (5,1,2), (3,1,1) )$ & with respective contents & $2,1,0,-1 $ \, ,\\
 $X_{7}=( (3,1,2) , (4,1,2), (5,1,2), (6,1,2) )$ & with respective contents & $2,1,0,-1 $ \, ,\\ 
 $X_{8}=( (5,1,3) , (4,1,2), (5,1,2), (3,1,1) )$ & with respective contents & $2,1,0,-1 $ \, ,\\
 $X_{9}=( (5,1,3) , (4,1,2), (5,1,2), (6,1,2) )$ & with respective contents & $2,1,0,-1 $ \, ,\\
 $X_{10}=( (5,1,3) , (6,1,3), (5,1,2), (3,1,1) )$ & with respective contents & $2,1,0,-1 $ \, ,\\  
 $X_{11}=( (5,1,3) , (6,1,3), (5,1,2), (6,1,2) )$ & with respective contents & $2,1,0,-1 $ \, ,\\  
 $X_{12}=( (5,1,3) , (6,1,3), (7,1,3), (3,1,1) )$ & with respective contents & $2,1,0,-1 $ \, ,\\  
 $X_{13}=( (5,1,3) , (6,1,3), (7,1,3), (6,1,2) )$ & with respective contents & $2,1,0,-1 $ \, ,\\
 $X_{14}=( (5,1,3) , (6,1,3), (7,1,3), (8,1,3) )$ & with respective contents & $2,1,0,-1 $ \, ,\\
\end{tabular}
\end{center}
and so on.
\end{exa}

\subsection{Action of the Heisenberg crystal operators $\tb_{\si}$}

We will define an order on the vertical $e$-strips of a given charged multipartition.
First, 
let $\ga=(a,b,j)$ and $\ga'=(a',b',j')$ be two boxes of $|\bla,\bs\rangle$.
Write $\ga>\ga'$ if
$\fc(\ga)>\fc(\ga')$ or $\fc(\ga)=\fc(\ga')$ and $j<j'$.

\begin{rem}
Note that this is the total order used to define the good boxes in a charged $\ell$-partition,
which characterises the Kashiwara crystals, see \cite[Chapter 6]{GeckJacon2011}.
\end{rem}

By extension, let $>$ denote the lexicographic order induced by $>$ on $e$-tuples of boxes in a given charged $\ell$-partition.
This restricts to a total order on $\sV(\bla,\bs)$.

\begin{defi}\label{defgoodverticalstrip}
Let $|\bla,\bs\rangle$ be a charged $\ell$-partition. Denote simply $\sV=\sV(\bla,\bs)$.
\begin{itemize}
\item 
The first good vertical $e$-strip of $|\bla,\bs\rangle$ 
is the maximal element $X_1$ of $\sV$ with respect to $>$.

\item Let $k\geq 2$. The $k$-th good vertical $e$-strip of $|\bla,\bs\rangle$ 
is the maximal element $X_k$ of 
$$\left\{ \, X\in\sV \, | \, X_{k-1}>X\mand X_{k-1}\cap X=\emptyset \, \right\} $$ 
with respect to $>$.
\end{itemize}
\end{defi}

In other terms, the greatest (with respect to $>$) vertical strip of $|\bla,\bs\rangle$ is good, and
another admissible vertical strip is good except if one of its boxes already belongs to a previous good vertical strip.

\begin{exa}
We go back to Example \ref{exaverticalstrip}.
Then we have $X_k>A_{k-1}$ for all $k=1,\dots,14$.
Moreover, 
there are only four good addable vertical strips among these, 
namely $X_1$, $X_2$, $X_6$ and $X_{13}$.
\end{exa}

\begin{defi}\label{deftc} Let $\si=(\si_1,\si_2\dots)\in\Pi$.
Define $\tc_\si : B_s \ra  B_s$ 
$$\tc_\si |\bla,\bs\rangle = |\bmu,\bs\rangle 
$$
where $\bmu$ 
is obtained from $\bla$ by adding 
recursively $\si_k$ times the $k$-th good vertical $e$-strip for $k\geq1$.
\end{defi}

\begin{lem}\label{lemheis1}
The map $\tc_{(1)}$ is well defined.
\end{lem}

\proof
Let $|\bla,\bs\rangle$ be a charged multipartition.
We need to prove that the first good vertical strip $X$ of $|\bla,\bs\rangle$ is addable.
Assuming it is not the case, then there exists a box $(a,b,j)\in X$
such that $(a-1,b,j)\notin X\sqcup\bla$ and $(a-1,b-1,j)\in\bla$.
But $(a-1,b,j)\notin \bla$ and $(a-1,b-1,j)\in\bla$ implies $(a-1,b,j)\in X$,
whence a contradiction.
\endproof

\begin{cor}\label{corheis1}
For all $\si\in\Pi$, the map $\tc_\si$ is well defined.
\end{cor}

\proof
If $\si=\emptyset$, then $\tc_\si=\id$ and so is well defined. 
So we assume $\si\neq\emptyset$, say $\si=(\si_1,\si_2,\dots)$.
First of all, Lemma \ref{lemheis1} and the definition of $\tc_{\si}$ implies that
$\tc_{(n)}$ is well defined for all $n\in\Z_{\geq1}$ (so in particular for $n=\si_1$).
It remains to prove that for all $\bla\in\Pi_\ell$ and for all $\ell$-charge $\bs$,
the second good vertical $e$-strip $X$ of $\tc_{(\si_1)}|\bla,\bs\rangle$ is addable.
Since $\si_1\geq1$, it is clear that for all $(a,b,j)\in X$,
$(a-1,b,j)$ is actually a box of $\bla$,
since if it was not, $X$ could not be the second good vertical strip of $\bla$
exactly as in the proof of Lemma \ref{lemheis1}.
Iterating, for all partition $\tau$, the $(h+1)$-th vertical strip of $\tc_\tau|\bla,\bs\rangle$
is addable, where $h$ is the number of non-zero parts in $\tau$.
\endproof

\ytableausetup{boxsize=1.2mm}

\begin{exa}
Take again the values of Example \ref{exaverticalstrip}.
Then, for $\si=(2.1^3)$, we have
$$
\tc_\si|\bla,\bs\rangle = 
\tc_{\Yboxdim{12pt}
\begin{ytableau}
*(Apricot) & *(NavyBlue)\\
*(YellowGreen)\\
*(Goldenrod) \\
*(Gray) 
\end{ytableau}
}
\,
{\tiny\left(\,\young(1234,01) \, , \, \young(45) \, \, , \, \young(67,56,4,3) \, \right) }
=
\ytableausetup{boxsize=3.1mm}
\left( \,
\scalebox{1}{
{\tiny
\begin{ytableau}
1 & 2 & 3 & 4 & *(Apricot) 5 & *(NavyBlue) 6 \\
0 & 1 & *(YellowGreen) 2  \\
*(Goldenrod) \text{-}1 \\
\end{ytableau}
}}
\, , \, 
\scalebox{1}{
{\tiny
\begin{ytableau}
4 & 5 & *(Apricot) 6 & *(NavyBlue) 7 \\
*(YellowGreen) 3  \\
*(Goldenrod) 2 \\
*(Goldenrod) 1 \\
*(Goldenrod) 0 \\
*(Gray) \text{-}1 \\
\end{ytableau}
}}
\, , \,
\scalebox{1}{
{\tiny
\begin{ytableau}
6 & 7 & *(Apricot) 8 & *(NavyBlue) 9 \\
5 & 6 & *(Apricot) 7 & *(NavyBlue) 8 \\
4 & *(YellowGreen) 5  \\
3 & *(YellowGreen) 4  \\
*(Gray) 2 \\
*(Gray) 1 \\
*(Gray) 0 \\
\end{ytableau}
}}
\, \right).
$$
Each box of $\si$ corresponds to a vertical $e$-strip, the matching being given by the colors.
\end{exa}

We will prove the following theorem.

\begin{thm}\label{thmheis}
For all $\si\in\Pi$, we have 
$$ \tc_\si =\tb_\si.
$$
\end{thm}

\subsection{Proof of Theorem \ref{thmheis}}

The strategy for proving this result consists in starting
from the doubly highest weight case, in which we know an explicit formula for $\tb_{\si}$.
Then, we show that the $\tc_{\pm\si}$ commute with the Kashiwara crystal operators for $\Ul$, then $\Ue$,
and use the commutation of the Kashiwara crystals and the $\cH$-crystal (see Definition \ref{defbsigma}) to conclude.

\begin{prop}\label{propheis1}
Let $|\bla,\bs\rangle$ be a doubly highest weight vertex.
Then, for all $\si\in\Pi$, we have
$$ 
\tc_\si|\bla,\bs\rangle =\tb_\si|\bla,\bs\rangle.
$$
\end{prop}

\proof
We need to translate the explicit formula for $\tb_{\si}$, given in terms of abaci, in terms of $\ell$-partitions.
By the correspondence given in Section \ref{mpwedge},
an $e$-period  in the $\ell$-abacus $\cA=\cA(\bla,\bs)$ (see \cite[Section 2.3]{JaconLecouvey2012})
corresponds to a good vertical $e$-strip in $|\bla,\bs\rangle$.
Therefore, it yields
an addable admissible vertical $e$-strip if $(j_1,d_1+1)\notin\cA$ where $(j_1,d_1)$ is the first bead of the period.
Thus, shifting the $k$-th $e$-period of 
$\cA$ by $\si_k$ steps to the right 
amounts to adding 
the $k$-th good vertical strip of $|\bla,\bs\rangle$.
In other terms $\tb_{\si}$ is the same as $\tc_{\si}$ when restricted to doubly highest weight vertices.
\endproof

\begin{prop}\label{propheis2}
Let $|\bla,\bs\rangle$ be a highest weight vertex for $\Ue$.
Then, for all $\si\in\Pi$, we have
$$ 
\tc_\si|\bla,\bs\rangle =\tb_\si|\bla,\bs\rangle.
$$
\end{prop}

\newcommand{\bbla}{\overline{\bla}}
\newcommand{\bbs}{\overline{\bs}}

\proof
Write $|\bla,\bs\rangle = \dot{F_{\bj}} |\bbla,\bbs\rangle$
where $|\bbla,\bbs\rangle$ is the highest weight vertex for $\Ul$ associated to $|\bla,\bs\rangle$,
and where $\dot{F_{\bj}} = \tdf_{j_r}\dots\tdf_{j_1}$ is a sequence of Kashiwara crystal operators of $\Ul$.
Because of Theorem \ref{thmlevelrank}, the two Kashiwara crystals commute, and
thus $|\bbla,\bbs\rangle$ is a doubly highest weight vertex.
We prove the result by induction on $r\in\N$.
If $r=0$, then $|\bla,\bs\rangle$ is a doubly highest weight vertex and this is Proposition \ref{propheis1}.
Suppose that the result holds for a fixed $r-1\geq0$.
Write $|\bnu,\bt\rangle=\tdf_{j_{r-1}}\dots\tdf_{j_1}|\bbla,\bbs\rangle$, so that
$|\bla,\bs\rangle = \tdf_{j_r} |\bnu,\bt\rangle$.
Because the crystal level-rank duality is realised in terms of abaci, 
we need to investigate one last time the action of $\tdf_{j_r}$ in the abacus.
We know that $\tdf_{j_r}$ acts on an $e$-partition by adding its good addable $j_r$-box 
(i.e. of content $j_r$ modulo $\ell$).
This corresponds to shifting a particular (white) bead one step up in the $e$-abacus of $|\bnu,\bs\rangle$,
see Example \ref{exalevelrank} for an illustration.
Now, if $j_r\neq 0$, 
then this corresponds to moving down a (black) bead in the $\ell$-abacus.
Since the resulting abacus $\cA(\bla,\bs)$ is again totally $e$-periodic
(the two Kashiwara crystals commute), this preserves the $e$-period containing this bead.
If $j_r=0$, then moving this white bead up corresponds to moving a black bead in position $(\ell,d)$ in the $\ell$-abacus
(which is the first element of its $e$-period) down to position $(1,d-e)$. Again, this preserves the $e$-period.
In both cases, the reduced $j_r$-word in the $e$-abacus (see \cite[Section 4.2]{Gerber2016} for details)
is unchanged and 
\begin{equation}\tag{$\ast$}\label{eq1}
 \tdf_{j_r} \tc_{\si}  |\bnu,\bt\rangle =   \tc_{\si} \tdf_{j_r}  |\bnu,\bt\rangle.
\end{equation}
Therefore, we have 
$$\arraycolsep=1.4pt\def\arraystretch{1.5}
\begin{array}{rcll}
\tb_{\si} |\bla,\bs\rangle & = &  \tb_{\si} \tdf_{j_r}|\bnu,\bt\rangle & \\
& = &  \tdf_{j_r} \tb_{\si}  |\bnu,\bt\rangle & \text{by Definition \ref{defbsigma}} \\
& = &  \tdf_{j_r} \tc_{\si}  |\bnu,\bt\rangle & \text{by induction hypothesis}\\
& = &   \tc_{\si} \tdf_{j_r}  |\bnu,\bt\rangle & \text{by (\ref{eq1})}\\
& = &   \tc_{\si} |\bla,\bs\rangle. & \\
\end{array}
$$
\endproof

We are now ready to prove Theorem \ref{thmheis}.
It remains to investigate the action of the Kashiwara crystal operators of $\Ue$.

\medskip

\textit{Proof of Theorem \ref{thmheis}.}
Write $|\bla,\bs\rangle = F_\bi|\bbla,\bs\rangle$
where $|\bbla,\bs\rangle$ is the highest weight vertex for $\Ue$ associated to $|\bla,\bs\rangle$,
and where ${F_{\bi}} = \tf_{i_r}\dots\tf_{i_1}$ is a sequence of Kashiwara crystal operators of $\Ue$.
We prove the result by induction on $r\in\N$.
If $r=0$, then $|\bla,\bs\rangle$ is a highest weight vertex for $\Ue$ and this is Proposition \ref{propheis2}.
Suppose that the result holds for a fixed $r-1\geq0$.
Write $|\bnu,\bs\rangle=\tf_{i_{r-1}}\dots\tf_{i_1}|\bbla,\bs\rangle$, so that
$|\bla,\bs\rangle = \tf_{i_r} |\bnu,\bs\rangle$.
Consider the reduced $i_r$-word for $|\bnu,\bs\rangle$.
Again, it is preserved by the action of $\tc_\si$
by Property (1) of Definition \ref{defvertstrip},
and we have
$$
\tdf_{j_r} \tc_{\si}  |\bnu,\bt\rangle =   \tc_{\si} \tdf_{j_r}  |\bnu,\bt\rangle.
$$
The commutation of $\tb_\si$ with $\tf_{i_r}$ together with the induction hypothesis completes the proof.
\begin{flushright}
$\square$
\end{flushright}

\begin{rem}\label{remheis}
Theorem \ref{thmheis} also yields an explicit description of the operators $\tb_{1,c}$.
It acts on any charged $\ell$-partition by adding its $k$-th good vertical $e$-strip,
where $c$ is the content of the $k$-th addable box of $\ka$ (with respect to the order $>$ on boxes).
\end{rem}

\subsection{Impact of conventions and relations with other results}

We end this section by mentioning an alternative realisation of the $\cH$-crystal.
In fact, some of the combinatorial procedures require some conventional choices, such as 
the maps $\tau$ and $\dtau$ yielding the level-rank duality, or
the order on the boxes or on the
vertical strips of multipartition.
Like in the case of Kashiwara crystals, see e.g. \cite[Remark 3.17]{BrundanKleshchev2009} or \cite[Remark 4.9]{Gerber2016}
Choosing a different convention yields to a slightly different version of the Heisenberg crystal.

We have already seen in the proof of Theorem \ref{thmschur} that conventions can be adjusted to 
fit Losev's \cite{Losev2015} or Iijima's \cite{Iijima2012} results about the action of a Heisenberg crystal operator
or a Schur function respectively. This is done by using the conjugation of multipartitions.
More precisely, one can decide to identify a charged $\ell$-partition $|\bla,\bs\rangle$ with $|\la',s\rangle$ instead of $|\la,s\rangle$
with the notations of Section \ref{abaci}.
Then, the action of the Heisenberg crystal operators is expressed in terms of addable
horizontal $e$-strips in the $\ell$-partition.
This is equivalent to changing the order on the vertical strips, applying a Heisenberg crystal operator, and then conjugate.

\medskip

One could also decide to exchange the role of $\tau$ and $\dtau$ in the level-rank duality, see Appendix.
Then, applying a Heisenberg crystal operator on an $\ell$-abacus would de decribed in terms of the corresponding $e$-abacus.
More precisely, $\tb_{1,c}$ would then consists in shifting a $\ell$-period in the $e$-abacus in the particular case
where it is totally $\ell$-periodic (see also the proof of Proposition \ref{propheis1}).
This permits to give an interpretation of Tingley's tightening procedure on descending $\ell$-abaci
\cite[Definition 3.8]{Tingley2008}, thereby answering Question 1 of \cite[Section 6]{Tingley2008}.

\begin{prop}\label{proptighten}
Let $\cA$ be a descending $\ell$-abacus. For all $k\geq1$, denote $T_k$ the $k$-th tightening operator associated to $\cA$.
We have $$T_k (\cA) = \tb_{-1,c}(\cA)$$ 
where $c$ is determined by $k$.
\end{prop}

\proof Recall that for this statement, we have considered the realisation of the Heisenberg crystal with
the alternative version of level-rank-duality, swapping the roles of $\tau$ and $\dtau$,
so that the operators $\tb_{-1,c}$ act on the $e$-abacus (rather than the $\ell$-abacus) by removing
a vertical $\ell$-strip (Theorem \ref{thmheis}).
Remember that $\tb_{-1,c} = \tb_{\theta} \tb_{-\ka}$ where $\theta$ depends on $c$.
Now, if an $\ell$-abacus is descending \cite[Definition 3.6]{Tingley2008}, its corresponding $e$-abacus is
totally $\ell$-periodic.
In particular, the corresponding $e$-partition is a highest weight vertex in the $\Ul$-crystal, and we can use 
Proposition \ref{propheis2}. 
In particular, $\tb_{-\ka}$ and $\tb_{\theta}$ act by shifting $\ell$-periods in the $e$-abacus, 
and $\tb_{-1,c}$ shifts one $\ell$-period, say the $k$-th one, one step to the left in the $e$-abacus.
This precisely what $T_k$ does.
\endproof

Finally, we mention that in another particular case, the Heisenberg crystal operator $\tb_{-\ka}$ coincide with 
the canonical $\Ue$-isomorphism $\varphi$ (up to cyclage) of \cite[Theorem 5.26]{Gerber2015} 
used to construct an affine Robinson-Schensted correspondence.

\begin{prop}\label{propreduc}
Let $|\bla,\bs\rangle\in B_s$ be a doubly highest weight vertex. Write $\ka=\ka|\bla,\bs\rangle$. We have
$$\tb_{-\ka}\xi^{h} |\bla,\bs\rangle = \varphi|\bla,\bs\rangle,$$
where $h$ is the number of parts of $\ka$ and 
$\xi$ is the cyclage isomorphism, see \cite[Proposition 4.4]{Gerber2015}.
\end{prop}

\proof
We use the notations of \cite{Gerber2015}.
First of all, because of \cite[Proposition 5.7]{Gerber2016}, 
doubly highest weight vertices are cylindric in the sense of \cite[Definition 2.3]{Gerber2015},
and $\varphi|\bla,\bs\rangle$ is therefore well defined, and simply verifies $\varphi=\psi^t$ where $t$
is the number of pseudoperiods in $|\bla,\bs\rangle$
and $\psi$ is the reduction isomorphism.
In fact, by definition of $\ka$, we have $t=h$, the number of (non-zero) parts of $\ka$,
and it suffices to apply the cyclage $h$ times to $|\bla,\bs\rangle$ to
match the formulas for $\tb_{-\ka}$ and $\psi^t$.
\endproof

% \begin{rem}
% It should be possible to weaken the condition on $|\bla,\bs\rangle$ and state this result for multipartitions
% that are highest weight for $\Ul$ only.
% \end{rem}

\subsection{Examples of computations}

By Remark \ref{remheis}, the Heisenberg crystal can be computed recursively from
its highest weight vertices, each of which yields a unique connected component,
isomorphic to the Young graph by Proposition \ref{propHcrystal}.

\medskip

The empty multipartition is obviously a highest weight vertex for $\cH$,
and so is every multipartition with less than $e$ boxes.
For instance, if $\ell=2$, $\bs=(0,1)$ and $e=3$, we can compute the connected
components of the Heisenberg crystal of $\cF_\bs$ with highest weight vertex $(\m,\m)$
and $\left(  {\tiny \young(0)} \ , \m \right)$. Up to rank $13$, we get the following subgraph
of the $\cH$-crystal.

\begin{figure}[H] 

\makebox[\textwidth][c]{

\scalebox{.65}{

\begin{tabular}{cc}

\begin{tikzpicture}[scale=0.6, every node/.style={scale=1}]

\node (a) at (8,12) {%
${ \tiny  \m \, \m }$
};

\node (b) at (8,8) {%
{${ \tiny \young(0,\moinsun) \, \young(1)   }$}
};

\node (c1) at (5,4) {%
{${ \tiny \young(01,\moinsun0) \, \young(12) }$}
}; 
\node (c2) at (11,4) {%
{${ \tiny \young(0,\moinsun,\moinsdeux) \, \young(1,0,\moinsun) }$}
};

\node(d1) at (2,0) {%
{${ \tiny \young(012,\moinsun01) \, \young(123)  }$}
};
\node(d2) at (8,0) {%
{${ \tiny \young(01,\moinsun0,\moinsdeux) \, \young(12,0,\moinsun) }$}
};
\node(d3) at (14,0) {%
{${ \tiny \young(0,\moinsun,\moinsdeux,\moinstrois,\moinsquatre) \, \young(1,0,\moinsun,\moinsdeux) }$}
};

\node(e1) at (0,-4) {%
{${ \tiny \young(0123,\moinsun012) \, \young(1234)   }$}
};
\node(e2) at (4,-4) {%
{${ \tiny \young(012,\moinsun01,\moinsdeux) \, \young(123,0,\moinsun)  }$}
};
\node(e3) at (8,-4) {%
{${ \tiny \young(01,\moinsun0,\moinsdeux\moinsun) \, \young(12,01,\moinsun0) }$}
};
\node(e4) at (12,-4) {%
{${ \tiny \young(01,\moinsun0,\moinsdeux,\moinstrois,\moinsquatre) \, \young(12,0,\moinsun,\moinsdeux)  }$}
};
\node(e5) at (16,-4) {%
{${ \tiny \young(0,\moinsun,\moinsdeux,\moinstrois,\moinsquatre,\moinscinq) \, \young(1,0,\moinsun,\moinsdeux,\moinstrois,\moinsquatre) }$}
};

\draw[->] (a) -- node [font=\tiny, left] {} (b) ;
\draw[->] (b) -- node [font=\tiny,left] {} (c1) ;
\draw[->] (b) -- node [font=\tiny, above] {} (c2) ;
\draw[->] (c1) -- node [font=\tiny, above] {} (d1) ;
\draw[->] (c1) -- node [font=\tiny, above] {} (d2) ;
\draw[->] (c2) -- node [font=\tiny] {} (d3) ;
\draw[->] (c2) -- node [font=\tiny, above] {} (d2) ;

\draw[->] (d1) -- node [font=\tiny, above] {} (e1) ;
\draw[->] (d1) -- node [font=\tiny, above] {} (e2) ;
\draw[->] (d2) -- node [font=\tiny, above] {} (e2) ;
\draw[->] (d2) -- node [font=\tiny, above] {} (e3) ;
\draw[->] (d2) -- node [font=\tiny, above] {} (e4) ;
\draw[->] (d3) -- node [font=\tiny, above] {} (e4) ;
\draw[->] (d3) -- node [font=\tiny, above] {} (e5) ;

\node (f) at (0,-7) { \ };

\end{tikzpicture}

&

\begin{tikzpicture}[scale=0.6, every node/.style={scale=1}]

\node (a) at (8,12) {%
${ \tiny  \young(0) \, \m }$
};

\node (b) at (8,8) {%
{${ \tiny \young(0,\moinsun) \, \young(10)   }$}
};

\node (c1) at (5,4) {%
{${ \tiny \young(01,\moinsun0) \, \young(12,0) }$}
}; 
\node (c2) at (11,4) {%
{${ \tiny \young(0,\moinsun,\moinsdeux,\moinstrois) \, \young(1,0,\moinsun) }$}
};

\node(d1) at (2,0) {%
{${ \tiny \young(012,\moinsun01) \, \young(123,0)  }$}
};
\node(d2) at (8,0) {%
{${ \tiny \young(01,\moinsun0,\moinsdeux,\moinstrois) \, \young(12,0,\moinsun) }$}
};
\node(d3) at (14,0) {%
{${ \tiny \young(0,\moinsun,\moinsdeux,\moinstrois,\moinsquatre) \, \young(1,0,\moinsun,\moinsdeux,\moinstrois) }$}
};

\node(e1) at (0,-4) {%
{${ \tiny \young(0123,\moinsun012) \, \young(1234,0)   }$}
};
\node(e2) at (4,-4) {%
{${ \tiny \young(012,\moinsun01,\moinsdeux,\moinstrois) \, \young(123,0,\moinsun)  }$}
};
\node(e3) at (8,-4) {%
{${ \tiny \young(01,\moinsun0,\moinsdeux\moinsun,\moinstrois) \, \young(12,01,\moinsun0) }$}
};
\node(e4) at (12,-4) {%
{${ \tiny \young(01,\moinsun0,\moinsdeux,\moinstrois,\moinsquatre) \, \young(12,0,\moinsun,\moinsdeux,\moinstrois)  }$}
};
\node(e5) at (16,-4) {%
{${ \tiny \young(0,\moinsun,\moinsdeux,\moinstrois,\moinsquatre,\moinscinq,\moinssix) \, \young(1,0,\moinsun,\moinsdeux,\moinstrois,\moinsquatre) }$}
};

\draw[->] (a) -- node [font=\tiny, left] {} (b) ;
\draw[->] (b) -- node [font=\tiny,left] {} (c1) ;
\draw[->] (b) -- node [font=\tiny, above] {} (c2) ;
\draw[->] (c1) -- node [font=\tiny, above] {} (d1) ;
\draw[->] (c1) -- node [font=\tiny, above] {} (d2) ;
\draw[->] (c2) -- node [font=\tiny] {} (d3) ;
\draw[->] (c2) -- node [font=\tiny, above] {} (d2) ;

\draw[->] (d1) -- node [font=\tiny, above] {} (e1) ;
\draw[->] (d1) -- node [font=\tiny, above] {} (e2) ;
\draw[->] (d2) -- node [font=\tiny, above] {} (e2) ;
\draw[->] (d2) -- node [font=\tiny, above] {} (e3) ;
\draw[->] (d2) -- node [font=\tiny, above] {} (e4) ;
\draw[->] (d3) -- node [font=\tiny, above] {} (e4) ;
\draw[->] (d3) -- node [font=\tiny, above] {} (e5) ;

\end{tikzpicture}
\end{tabular}
}
}
% \caption{Two connected components of the Heisenberg crystal of the level 2 Fock space
% with parameters $\bs=(0,1)$ and $e=3$ up to rank 13.}
\label{ex1}
\end{figure}

For $e=2$ and $\bs=(3,2,5)$, the $3$-partition $(2,1,4)$ is clearly
a highest weight vertex for $\cH$ by Theorem \ref{thmheis}.
The corresponding connected component, up to rank $17$ is the following graph.

\begin{figure}[H] 
\makebox[\textwidth][c]{
\scalebox{.65}{

\begin{tikzpicture}[scale=0.6, every node/.style={scale=1}]

\node (a) at (8,12) {%
${ \tiny  \young(34) \, \young(2) \, \young(5678)   }$
};

\node (b) at (8,8) {%
{${ \tiny \young(34) \, \young(23) \, \young(5678,4)  }$}
};

\node (c1) at (4,4) {%
{${ \tiny \young(345) \, \young(234) \, \young(5678,45)}$}
}; 
\node (c2) at (12,4) {%
{${ \tiny\young(34,2) \, \young(23) \, \young(5678,4,3) }$}
};

\node(d1) at (0,0) {%
{${ \tiny \young(345) \, \young(234) \, \young(5678,456)}$}
};
\node(d2) at (8,0) {%
{${ \tiny \young(34,2) \, \young(234) \, \young(5678,45,3)}$}
}; 
\node(d3) at (16,0) {%
{${ \tiny \young(34,2,1) \, \young(23) \, \young(5678,4,3,2)}$}
};

\node(e1) at (-6,-4) {%
{${ \tiny \young(3456) \, \young(234) \, \young(5678,4567)}$}
}; 
\node(e2) at (1,-4) {%
{${ \tiny \young(345,2) \, \young(234) \, \young(5678,456,3)}$}
}; 
\node(e3) at (8,-4) {%
{${ \tiny \young(34,23) \, \young(234) \, \young(5678,45,34)}$}
}; 
\node(e4) at (15,-4) {%
{${ \tiny \young(34,2,1,0) \, \young(234) \, \young(5678,45,3,2)}$}
}; 
\node(e5) at (22,-4) {%
{${ \tiny \young(34,2,1,0) \, \young(23,1) \, \young(5678,4,3,2)}$}
};

\draw[->] (a) -- node [font=\tiny, left] {} (b) ;
\draw[->] (b) -- node [font=\tiny,left] {} (c1) ;
\draw[->] (b) -- node [font=\tiny, above] {} (c2) ;
\draw[->] (c1) -- node [font=\tiny, above] {} (d1) ;
\draw[->] (c1) -- node [font=\tiny, above] {} (d2) ;
\draw[->] (c2) -- node [font=\tiny] {} (d3) ;
\draw[->] (c2) -- node [font=\tiny, above] {} (d2) ;

\draw[->] (d1) -- node [font=\tiny, above] {} (e1) ;
\draw[->] (d1) -- node [font=\tiny, above] {} (e2) ;
\draw[->] (d2) -- node [font=\tiny, above] {} (e2) ;
\draw[->] (d2) -- node [font=\tiny, above] {} (e3) ;
\draw[->] (d2) -- node [font=\tiny, above] {} (e4) ;
\draw[->] (d3) -- node [font=\tiny, above] {} (e4) ;
\draw[->] (d3) -- node [font=\tiny, above] {} (e5) ;

\node (f) at (0,-7) { \ };

\end{tikzpicture}
}
}
% \caption{}
% \label{}
\end{figure}

We also wish to give an example of the asymptotic case.
Take $\ell=3$, $\bs=(0,7,19)$, $\bla=(1,3.2.1,3.1)$ and $e=3$,
so that $|\bla,\bs\rangle$ is a highest weight vertex for $\cH$ and is $3$-asymptotic.
We see in the following corresponding $\cH$-crystal that 
the elements $\tb_{\si}$, for $|\si|< 4$, only act on the third component of $|\bla,\bs\rangle$,
but that $\tb_{(1^4)}$ acts already on the second component.
This illustrates Theorem \ref{thmschur} (2).

\begin{figure}[H] 
\makebox[\textwidth][c]{
\scalebox{.65}{

\begin{tikzpicture}[scale=1, every node/.style={scale=1}]

\node (a) at (8,12) {%
${ \tiny  \young(0) \, \young(789,67,5) \, \young(\dixneuf\vingt\vingtun,\dixhuit)   }$
};

\node (b) at (8,8) {%
${ \tiny  \young(0) \, \young(789,67,5) \, \young(\dixneuf\vingt\vingtun,\dixhuit,\dixsept,\seize,\quinze)   }$
};

\node (c1) at (4,4) {%
${ \tiny  \young(0) \, \young(789,67,5) \, \young(\dixneuf\vingt\vingtun,\dixhuit\dixneuf,\dixsept\dixhuit,\seize\dixsept,\quinze)   }$
}; 
\node (c2) at (12,4) {%
${ \tiny  \young(0) \, \young(789,67,5) \, \young(\dixneuf\vingt\vingtun,\dixhuit,\dixsept,\seize,\quinze,\quatorze,\treize,\douze)   }$
};

\node(d1) at (0,0) {%
${ \tiny  \young(0) \, \young(789,67,5) \, \young(\dixneuf\vingt\vingtun,\dixhuit\dixneuf\vingt,\dixsept\dixhuit\dixneuf,\seize\dixsept\dixhuit,\quinze)   }$
}; 
\node(d2) at (8,0) {%
${ \tiny  \young(0) \, \young(789,67,5) \, \young(\dixneuf\vingt\vingtun,\dixhuit\dixneuf,\dixsept\dixhuit,\seize\dixsept,\quinze,\quatorze,\treize,\douze)   }$
}; 
\node(d3) at (16,0) {%
${ \tiny  \young(0) \, \young(789,67,5) \, \young(\dixneuf\vingt\vingtun,\dixhuit,\dixsept,\seize,\quinze,\quatorze,\treize,\douze,\onze,\dix,9)   }$
};

\node(e1) at (-6,-4) {%
${ \tiny  \young(0) \, \young(789,67,5) \, \young(\dixneuf\vingt\vingtun\vingtdeux,\dixhuit\dixneuf\vingt\vingtun,\dixsept\dixhuit\dixneuf\vingt,\seize\dixsept\dixhuit,\quinze)   }$
}; 
\node(e2) at (1,-4) {%
${ \tiny  \young(0) \, \young(789,67,5) \, \young(\dixneuf\vingt\vingtun,\dixhuit\dixneuf\vingt,\dixsept\dixhuit\dixneuf,\seize\dixsept\dixhuit,\quinze,\quatorze,\treize,\douze)   }$
}; 
\node(e3) at (8,-4) {%
${ \tiny  \young(0) \, \young(789,67,5) \, \young(\dixneuf\vingt\vingtun,\dixhuit\dixneuf,\dixsept\dixhuit,\seize\dixsept,\quinze\seize,\quatorze\quinze,\treize\quatorze,\douze)   }$
}; 
\node(e4) at (15,-4) {%
${ \tiny  \young(0) \, \young(789,67,5) \, \young(\dixneuf\vingt\vingtun,\dixhuit\dixneuf,\dixsept\dixhuit,\seize\dixsept,\quinze,\quatorze,\treize,\douze,\onze,\dix,9)   }$
}; 
\node(e5) at (22,-4) {%
${ \tiny  \young(0) \, \young(789,67,56) \, \young(\dixneuf\vingt\vingtun,\dixhuit,\dixsept,\seize,\quinze,\quatorze,\treize,\douze,\onze,\dix,9,8,7)   }$
};

\draw[->] (a) -- node [font=\tiny, left] {} (b) ;
\draw[->] (b) -- node [font=\tiny,left] {} (c1) ;
\draw[->] (b) -- node [font=\tiny, above] {} (c2) ;
\draw[->] (c1) -- node [font=\tiny, above] {} (d1) ;
\draw[->] (c1) -- node [font=\tiny, above] {} (d2) ;
\draw[->] (c2) -- node [font=\tiny] {} (d3) ;
\draw[->] (c2) -- node [font=\tiny, above] {} (d2) ;

\draw[->] (d1) -- node [font=\tiny, above] {} (e1) ;
\draw[->] (d1) -- node [font=\tiny, above] {} (e2) ;
\draw[->] (d2) -- node [font=\tiny, above] {} (e2) ;
\draw[->] (d2) -- node [font=\tiny, above] {} (e3) ;
\draw[->] (d2) -- node [font=\tiny, above] {} (e4) ;
\draw[->] (d3) -- node [font=\tiny, above] {} (e4) ;
\draw[->] (d3) -- node [font=\tiny, above] {} (e5) ;

\node (f) at (0,-7) { \ };

\end{tikzpicture}
}
}
% \caption{}
% \label{}
\end{figure}

\appendix

\section{Crystal level-rank duality}\label{lr}

We recall, using a slightly different presentation, the results of \cite[Section 4]{Gerber2016} concerning the crystal version of level-rank duality.

\medskip

There is a double affine quantum group action on the wedge space.
In Section \ref{jmmofockspace}, we have explained how $\Ue$ acts on $\cF_s$.
It turns out that $\Ul$, where $p=-1/q$, acts on $\cF_s$ in a similar way. More precisely we will:
\begin{enumerate}
 \item index ordered wedges by charged $e$-partitions, using an alternative bijection $\dtau$,
 \item make $\Ul$ act on $\cF_s$ via this new indexation by swapping the roles of $e$ and $\ell$ and replacing $q$ by $p$.
\end{enumerate}

To define $\dtau$, recall that we have introduced the quantities $z(n)\in\Z$, $1\leq y(n) \leq \ell$ and $1\leq x(n)\leq e$ for each $n\in\Z$.
To each pair $(1,c)\in\{1\}\times\Z$, we associate the pair $\dtau(1,d)=(j,d)\in\{1,\dots,\ell\}\times\Z$
where
$$j=x(-c)\mand d=-(y(-c)-1)+\ell z(-c).$$

In particular, $\dtau$ sends the bead in position $(1,c)$ into
the rectangle $z(-c)$, on the row $x(-c)$ and column $y(-c)$ (numbered from right to left within each rectangle).

The $\dtau$ is bijective and we can see $\dtau^{-1}$ as the following procedure:
\begin{enumerate}
 \item Divide the $\ell$-abacus into rectangles of size $e\times\ell$, where the $z$-th rectangle ($z\in\Z$) contains the positions
 $(j,d)$ for all $1\leq j\leq \ell$ and $-e+1+ze\leq d \leq ze$.
 \item Relabel each $(j,d)$ by the second coordinate of $\dtau^{-1}(j,d)$, see Figure \ref{ab8} for an example.
 \item Replace the newly indexed beads on a $1$-abacus according to this new labeling.
\end{enumerate}

\begin{figure}[H]  \centering
\includegraphics[scale=1.2]{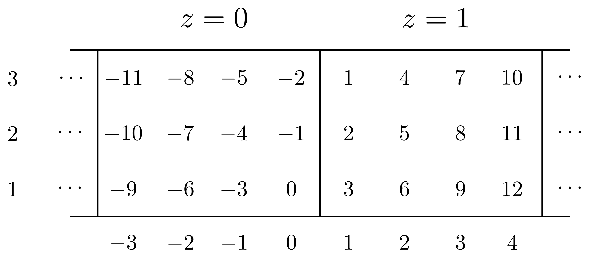}
\caption{Relabelling bead positions in the $e$-abacus according to $\dtau^{-1}$, for $\ell=4$ and $e=3$.}
\label{ab8}
\end{figure}

We see that $\dtau^{-1}$, so also $\dtau$, actually only depends on $e$, not on $\ell$.
In fact, explicit formulas for $\dtau$ and $\dtau^{(1)}$ are given by 
$$\dtau (1,c) = \left( (-c\mod e) + 1 \, , \, -\left\lfloor\frac{-c}{e}\right\rfloor \right) = \left( (-c\mod e) + 1 \, , \,  \frac{-c-(-c\mod e)}{e} \right), $$
and 
$$\dtau^{-1} (j,d) = \left( 1 \, , \, -(j-1)+ed \right),$$
and one notices that $\dtau$ corresponds to taking the usual $e$-quotient and the $e$-core of a partition, where the $e$-core is encoded in the $e$-charge.
More precisely, the renumbering of the beads according to $\dtau$ is the well-know ``folding'' procedure used to compute the $e$-quotient, see \cite{JamesKerber1984}.

\begin{defi}
The (twisted)  level-rank duality is the bijective map
$\dtau\circ(.)'\circ\tau^{-1}$ 
\end{defi}

\begin{rem}
The map $\dtau\circ\tau^{-1}$ already defines a level-rank duality, this was the one studied by Uglov \cite{Uglov1999}.
However, the twisted version defined above (i.e. where the conjugation is added) is the one that is relevant when it comes to crystals, see Theorem \ref{thmlevelrank} below.
\end{rem}

There is a convenient way to picture the crystal level-rank duality as follows.
Starting from an $\ell$-abacus,
stack copies on top of each other by translating by $e$ to the right.
Then, extract a vertical slice of the resulting picture, and read
off the corresponding $e$-partition by looking at the white beads (instead of black beads) 
in each column, starting from the leftmost one.

\begin{exa}\label{exalevelrank}
The abacus $\cA(\bla,\bs)$ with $\ell=4$, $e=3$, $\bla=(1,\emptyset,1^3,5)$ and $\bs=(-1,-1,1,1)$ looks as follows
(the origin is represented with the boldfaced vertical bar).
\begin{figure}[H]  \centering
\includegraphics[scale=0.8]{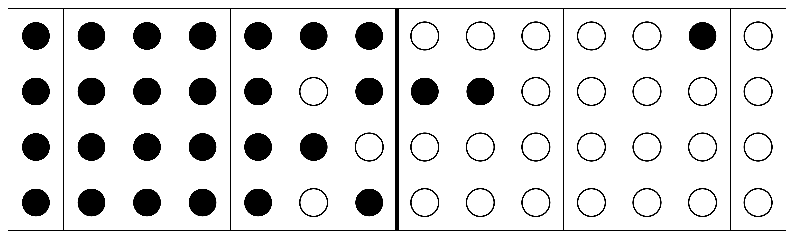}
\label{ab5}
\end{figure}
Stacking copies of $\cA(\bla,\bs)$ gives
\begin{figure}[H]  \centering
\includegraphics[scale=0.8]{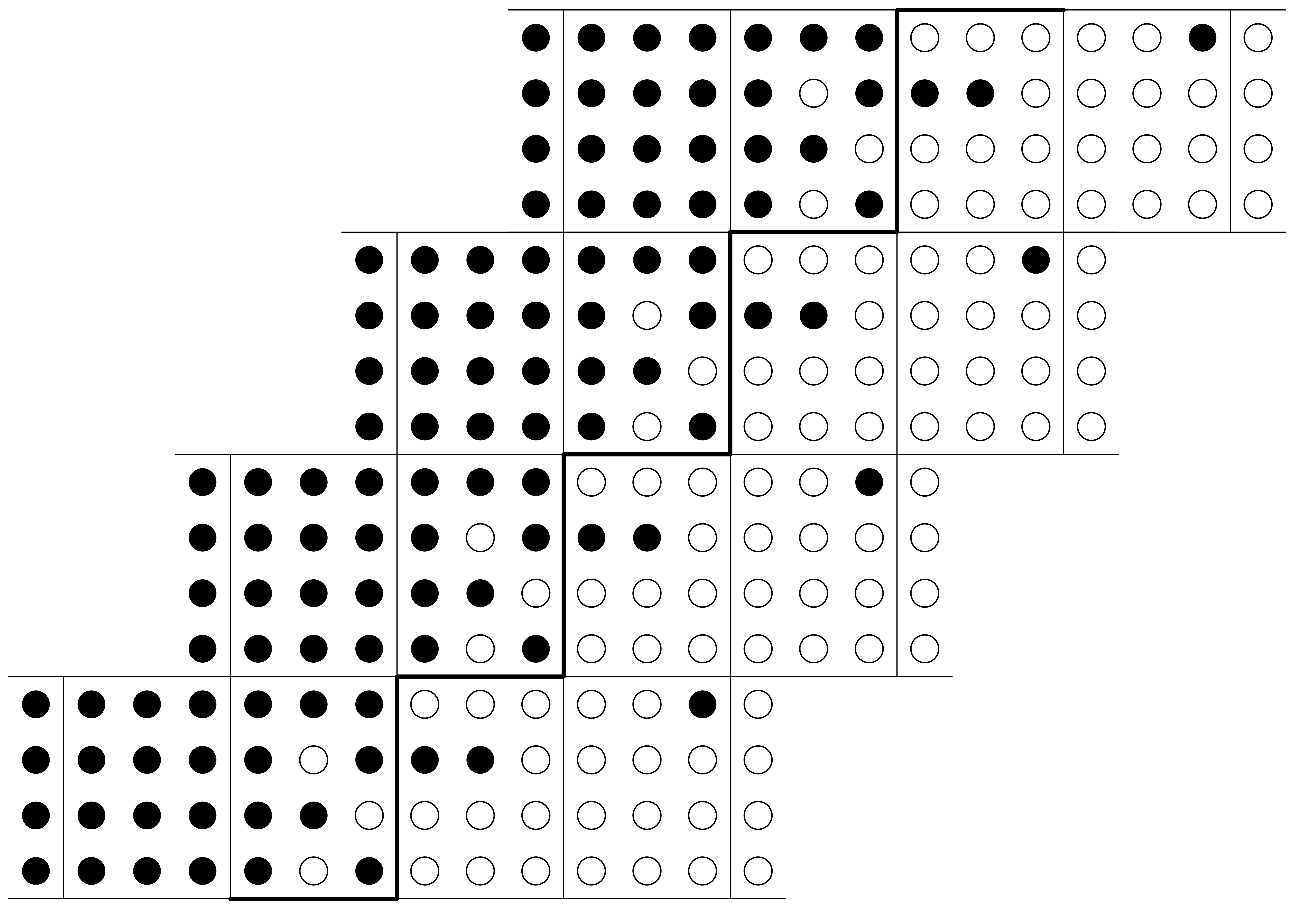}
\label{ab6}
\end{figure}
and extracting one vertical $e$-abacus yields
\begin{figure}[H]  \centering
\includegraphics[scale=0.8]{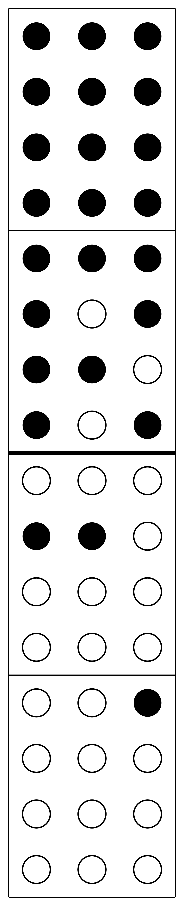}
\label{ab7}
\end{figure}
which corresponds to $|(1,2.1^2,2.1^4),(-1,1,0)\rangle$.
\end{exa}

We have the following induced maps
$$
\begin{array}{ccccccc}
\ds\bigoplus_{|\bs|=s}\cF_{\bs} & 
\xrightleftharpoons[\hspace{4mm}\tau\hspace{4mm}]{\tau^{-1}} & 
\cF_s & 
\xrightleftharpoons[\quad '\quad ]{'} & 
\cF_{-s} & 
\xrightleftharpoons[\dtau^{-1}]{\hspace{4mm}\dtau\hspace{4mm}}  & 
\ds\bigoplus_{|\dbs|=s} \cF_{\dbs}
\\
|\bla,\bs\rangle & \longleftrightarrow & 
|\la,s\rangle &
\longleftrightarrow & 
|\la',-s\rangle &
\longleftrightarrow &
|\dbla,\dbs\rangle
\end{array}
$$

Like $\tau$, the bijection $\dtau$ induces a $\Ul$-module isomorphism, and
$\cF_s$ has a $\Ul$-crystal, given by the same rule as the $\Ue$-crystal by swapping the roles of $e$ and $\ell$ and replacing $q$ by $p$.

\begin{thm}\label{thmlevelrank} Via level-rank duality,
\begin{enumerate}
 \item the $\Ue$-action and the $\Ul$-action on $\cF_s$ commute, and
\item the $\Ue$-crystal and the $\Ul$-crystal of $\cF_s$ commute.
\end{enumerate}
\end{thm}

\proof
The first point is essentially due to Uglov \cite[Proposition 4.6]{Uglov1999}, where he uses the non-twisted version of level-rank duality,
see \cite[Theorem 3.9]{Gerber2016} for the justification in the twisted case.
The second point is \cite[Theorem 4.8]{Gerber2016}.
\endproof

\textbf{Ackowledgments:} 
Many thanks to Emily Norton for pointing out an inconsistency in the first version of this paper and
for helpful conversations.

\bibliographystyle{plain}
% \bibliography{/users/gerber/Documents/Math/Recherche/Biblio/biblio}
% \bibliography{/home/thomas/Documents/Math/Research/Biblio/biblio}

\begin{thebibliography}{10}

\bibitem{Ariki1996}
Susumu Ariki.
\newblock {On the decomposition numbers of the {H}ecke algebra of
  {$G(m,1,n)$}}.
\newblock {\em J. Math. Kyoto Univ.}, 36(4):789--808, 1996.

\bibitem{BrundanKleshchev2009}
Jonathan Brundan and Alexander Kleshchev.
\newblock {Graded decomposition numbers for cyclotomic Hecke algebras}.
\newblock {\em Adv. Math.}, 222:1883--1942., 2009.

\bibitem{DudasVaragnoloVasserot2015}
Olivier Dudas, Michela Varagnolo, and \'Eric Vasserot.
\newblock {Categorical actions on unipotent representations I. Finite unitary
  groups}.
\newblock 2015.
\newblock arXiv:1509.03269.

\bibitem{DudasVaragnoloVasserot2016}
Olivier Dudas, Michela Varagnolo, and \'Eric Vasserot.
\newblock {Categorical actions on unipotent representations of finite classical
  groups}.
\newblock 2016.
\newblock To appear in \textit{Contemp. Math.}

\bibitem{Etingof2012}
Pavel Etingof.
\newblock {Supports of irreducible spherical representations of rational
  Cherednik algebras of finite Coxeter groups}.
\newblock {\em Adv. Math.}, 229:2042--2054, 2012.

\bibitem{FLOTW1999}
Omar Foda, Bernard Leclerc, Masato Okado, Jean-Yves Thibon, and Trevor Welsh.
\newblock {Branching functions of $A_{n-1}^{(1)}$ and Jantzen-Seitz problem for
  Ariki-Koike algebras}.
\newblock {\em Adv. Math.}, 141:322--365, 1999.

\bibitem{GeckJacon2011}
Meinolf Geck and Nicolas Jacon.
\newblock {\em {Representations of Hecke Algebras at Roots of Unity}}.
\newblock Springer, 2011.

\bibitem{Gerber2015}
Thomas Gerber.
\newblock {Crystal isomorphisms in Fock spaces and Schensted correspondence in
  affine type A}.
\newblock {\em Alg. and Rep. Theory}, 18:1009--1046, 2015.

\bibitem{Gerber2016}
Thomas Gerber.
\newblock {Triple crystal action in Fock spaces}.
\newblock 2016.
\newblock arXiv:1601.00581.

\bibitem{GerberHissJacon2015}
Thomas Gerber, Gerhard Hiss, and Nicolas Jacon.
\newblock {Harish-Chandra series in finite unitary groups and crystal graphs}.
\newblock {\em Int. Math. Res. Notices}, 22:12206--12250, 2015.

\bibitem{Iijima2012}
Kazuto Iijima.
\newblock { On a higher level extension of Leclerc-Thibon product theorem in
  $q$-deformed Fock spaces}.
\newblock {\em J. Alg.}, 371:105--131, 2012.

\bibitem{JaconLecouvey2012}
Nicolas Jacon and C\'edric Lecouvey.
\newblock {A combinatorial decomposition of higher level Fock spaces}.
\newblock {\em Osaka J. Math.}, 50(4):897--920, 2013.

\bibitem{JamesKerber1984}
Gordon James and Adalbert Kerber.
\newblock {\em {The Representation theory of the Symmetric Group}}.
\newblock Cambridge University Press, 1984.

\bibitem{JMMO1991}
Michio Jimbo, Kailash~C. Misra, Tetsuji Miwa, and Masato Okado.
\newblock Combinatorics of representations of {$U_q(\widehat{sl(n)})$} at
  {$q=0$}.
\newblock {\em Comm. Math. Phys.}, 136(3):543--566, 1991.

\bibitem{Kashiwara1993}
Masaki Kashiwara.
\newblock {Global crystal bases of quantum groups}.
\newblock {\em Duke Math. J.}, 69:455--485, 1993.

\bibitem{KashiwaraMiwaStern1995}
Masaki Kashiwara, Tetsuji Miwa, and Eugene Stern.
\newblock {Decomposition of $q$-deformed Fock spaces}.
\newblock {\em Selecta Math.}, 1:787--805, 1995.

\bibitem{LLT1997}
Alain Lascoux, Bernard Leclerc, and Jean-Yves Thibon.
\newblock {Ribbon Tableaux, Hall-Littlewood Functions, Quantum Affine Algebras
  and Unipotent Varieties}.
\newblock {\em J. Math. Phys.}, 38:1041--1068, 1997.

\bibitem{LeclercThibon2001}
Bernard Leclerc and Jean-Yves Thibon.
\newblock {Littlewood-Richardson coefficients and Kazhdan-Lusztig polynomials}.
\newblock In {\em Combinatorial Methods in Representation Theory}, volume~28 of
  {\em Advanced Studies in Pure Mathematics}. American Mathematical Society,
  2001.

\bibitem{Losev2015}
Ivan Losev.
\newblock {Supports of simple modules in cyclotomic Cherednik categories O}.
\newblock 2015.
\newblock arXiv:1509.00526.

\bibitem{Lusztig1989}
George Lusztig.
\newblock {Modular representations and quantum groups}.
\newblock {\em Contemp. Math.}, 82:58--77, 1989.

\bibitem{Macdonald1998}
Ian~G. Macdonald.
\newblock {\em {Symmetric functions and Hall polynomials}}.
\newblock Oxford Mathematical Monographs, second edition edition, 1998.

\bibitem{MiwaJimboDate2000}
Tetsuji Miwa, Michio Jimbo, and Etsuro Date.
\newblock {\em {Solitons: Differential Equations, Symmetries and Infinite
  Dimensional Algebras}}.
\newblock Cambridge University Press, 2000.

\bibitem{Shan2011}
Peng Shan.
\newblock {Crystals of Fock spaces and cyclotomic rational double affine Hecke
  algebras}.
\newblock {\em Ann. Sci. \'Ec. Norm. Sup\'er.}, 44:147--182, 2011.

\bibitem{ShanVasserot2012}
Peng Shan and \'Eric Vasserot.
\newblock {Heisenberg algebras and rational double affine Hecke algebras}.
\newblock {\em J. Amer. Math. Soc.}, 25:959--1031, 2012.

\bibitem{Stanley2001}
Richard~P. Stanley.
\newblock {\em {Enumerative Combinatorics}}, volume~2.
\newblock Cambridge University Press, 2001.

\bibitem{Tingley2008}
Peter Tingley.
\newblock {Three combinatorial models for $\widehat{\mathfrak{sl}_n}$ crystals,
  with applications to cylindric plane partitions}.
\newblock {\em Int. Math. Res. Notices}, Art. ID RNM143:1--40, 2008.

\bibitem{Uglov1999}
Denis Uglov.
\newblock {Canonical bases of higher-level $q$-deformed Fock spaces and
  Kazhdan-Lusztig polynomials}.
\newblock {\em Progr. Math.}, 191:249--299, 1999.

\bibitem{Yvonne2005}
Xavier Yvonne.
\newblock {\em {Bases canoniques d'espaces de Fock de niveau sup\'erieur}}.
\newblock PhD thesis, Universit\'e de Caen, 2005.

\end{thebibliography}

\end{document}